\documentclass[11pt]{amsart}

\usepackage{setspace} % allows control of line spacing
\usepackage{fullpage} % forces narrower page margins
\usepackage{lmodern} % allows for different fonts, e.g. Times
\usepackage{graphicx} % for loading graphics
\usepackage{color} % add color to text
\usepackage{float} % places float at exact location in LaTeX code
\usepackage{accents} % creating "fake" mathematica accents from non-accent symbols
\usepackage{enumitem} % better control of enumeration
\numberwithin{equation}{section}

\usepackage[numbers]{natbib}

\usepackage{amsmath, amsfonts, amssymb, amsthm} % standard AMS math pacakges
\usepackage{mathtools} % extension and fix of amsmath
\usepackage{xfrac} % allows for finer control of fraction size and placement
\usepackage{faktor} % enables command to type factor structures
\usepackage{bbm} % useful extension of mathbb to numbers, e.g. for indicator functions
\usepackage{mathrsfs} % provides mathscr option
\usepackage[mathscr]{euscript} % overrides mathrsfs by providing different scripted letters
\usepackage{upgreek} % upright lower case Greek letters
\newcommand{\R}{\mathbb{R}} % real numbers
\newcommand{\C}{\mathbb{C}} % complex numbers
\newcommand{\E}{\mathbb{E}} % expectation
\newcommand{\bigO}{\mathrm{O}} % big-O
\newcommand{\littleo}{\mathrm{o}} % little-O

\newcommand{\N}{\mathbb{N}} % negative definite functions
\newcommand{\M}{\mathcal{M}} % Mellin transform
\newcommand{\F}{\mathcal{F}} % Fourier transform
\newcommand{\Leb}{\mathrm{L}} % Lebesgue spaces
\newcommand{\Log}{\mathop{\rm {Log}}\nolimits}
\newcommand{\supp}{\mathop{\rm {supp}}\nolimits}

\renewcommand{\rm}[1]{\mathrm{#1}}
\renewcommand{\tilde}[1]{\widetilde{#1}}
\renewcommand{\epsilon}{\varepsilon}
\renewcommand{\geq}{\geqslant}
\renewcommand{\leq}{\leqslant}

\renewcommand{\Re}{\mathop{\rm {Re}}\nolimits}

\newcommand{\ap}{\mathscr{A}}
\newcommand{\apm}{\ap_\mathscr{D}}
\theoremstyle{plain}
\newtheorem{proposition}{Proposition}[section]
\newtheorem{lemma}{Lemma}[section]
\newtheorem{corollary}{Corollary}[section]
\newtheorem{theorem}{Theorem}[section]

\theoremstyle{definition}

\theoremstyle{remark}

\usepackage{hyperref}
\usepackage{cleveref}

\setlist[enumerate,1]{label=(\arabic*),ref=(\arabic*)}
\setlist[enumerate,2]{label=(\alph*),ref=(\arabic{enumi})(\alph*)}
\setlist[enumerate,3]{label=(\roman*),ref=(\arabic{enumi})(\alph{enumii})(\roman*)}
\setlist[enumerate,4]{label=(\Alph*),ref=(\arabic{enumi}-\alph{enumii}-\roman{enumiii}-\Alph*)}

\begin{document}

\title{Non-classical Tauberian and Abelian type criteria for the moment problem}

\author{P.~Patie}\thanks{The authors are grateful to an anonymous referee for her/his constructive comments and in particular for pointing out an issue in an earlier version of the proof of Corollary 1.2(1).  This work was partially supported by  NSF Grant DMS-1406599, a CNRS grant and the ARC IAPAS, a fund of the Communaut\'ee fran\c{c}aise de Belgique. Both authors are grateful  for the hospitality of  the  Laboratoire de Math\'ematiques et de leurs applications de Pau, where  this work was initiated.}
\address{School of Operations Research and Information Engineering, Cornell University, Ithaca, NY 14853.}
\email{pp396@cornell.edu}

\author{A.~Vaidyanathan}
\address{Center for Applied Mathematics, Cornell University, Ithaca, NY, 14853.}
\email{av395@cornell.edu}

\date{ \today }

% Revise abstract

\begin{abstract}
The aim of this paper is to provide some new criteria for the determinacy problem of the Stieltjes moment problem. We first give a Tauberian type criterion for moment indeterminacy that is expressed purely in terms of the asymptotic behavior of the moment sequence (and its extension to imaginary lines). Under an additional assumption this provides a converse to the classical Carleman's criterion, thus yielding an equivalent condition for moment determinacy. We also provide a criterion for moment determinacy that only involves the large asymptotic  behavior of the distribution (or of the density if it exists), which can be thought of as an Abelian counterpart to the previous Tauberian type result. This latter criterion generalizes Hardy's condition for determinacy, and under some further assumptions yields a converse to the Pedersen's refinement of the celebrated Krein's theorem. The proofs utilize non-classical Tauberian results for moment sequences that are analogues to the ones developed in \cite{feigin:1983} and \cite{balkema:1995} for the bi-lateral Laplace transforms in the context of asymptotically parabolic functions. We illustrate our results by studying the time-dependent moment problem for the law of log-L\'evy processes viewed as a generalization of the log-normal distribution. Along the way, we derive the large asymptotic behavior of the density of spectrally-negative L\'evy processes having a Gaussian component, which may be of independent interest.
\end{abstract}
\maketitle
%In this paper we investigate some new criteria for the Stieltjes moment problem. We give a criterion for moment indeterminacy purely in terms of conditions on the Mellin transform, which becomes a necessary and sufficient condition under some further assumptions. This criterion is particularly useful in the study of the time-dependent moment problem for processes of the form $(X_t)_{t \geq 0} = (e^{Y_t})_{t \geq 0}$, where $(Y_t)_{t \geq 0}$ is a suitable one-dimensional L\'evy process. We call this the log-L\'evy moment problem, and give an example of a process such that $X_t$ is moment determinate if and only if $t \leq 2$. We also provide a criterion for determinacy in terms of the behavior of the tail and show that this generalizes Hardy's condition for moment determinacy; when the measure admits a density we extend this to a necessary and sufficient condition. Along the way we also derive some interesting auxiliary results, such as the asymptotic at infinity, for all $t > 0$, of the densities of certain spectrally-negative L\'evy processes. The proofs make use of the Tauberian theory for asymptotically parabolic functions developed in \cite{feigin:1983,balkema:1995}.

\section{Introduction and Main Results} \label{sec:main-results}

The determinacy problem for the Stieltjes moment problem asks under what conditions a positive measure $\nu$ supported on $[0,\infty)$ can be uniquely determined by its sequence of moments $\M_\nu=\left(\M_\nu(n)\right)_{n \geq 0}$ where, for any $n\geq 0$,
\begin{equation*}
\M_\nu(n) = \int_0^\infty x^n \nu(dx)<\infty.
\end{equation*}
When a positive measure is uniquely determined by its moments we say it is moment determinate, otherwise it is moment indeterminate. Note that we consider only measures with unbounded support since otherwise the problem is trivial. For references on the moment problem see the classic monographs \cite{akhiezer:1965} and \cite{shohat:1943}, the comprehensive exposition \cite{simon:1998}, and the more recent monograph \cite{schmudgen:2017}, where the interested reader will find a nice description of its connections and interplay with many branches of mathematics, as well as its broad range of applications.

\subsection{A Tauberian type moment condition for indeterminacy, and a converse for Carleman's criterion}

% Revise introduction to this subsection

One of the most widely used criteria for determinacy is Carleman's criterion, which states that if
\begin{equation}\label{eq:def_ms}
\sum^\infty \M_\nu^{-\frac{1}{2n}}(n) = \infty
\end{equation}
then $\nu$ is moment determinate, where for a sequence $(a_n)_{n\geq 0}$ of real numbers  $\sum^\infty a_n= \infty$ denotes $\sum_{n_0}^\infty a_n = \infty$ for some index $n_0 \geq 1$ whose choice does not impact the divergence property (the same notation holds for integrals of functions). However, it is well-known that the divergence of this series is not a necessary condition for moment determinacy, see e.g.~Heyde \cite{heyde:1963} for an example. The main result in this section is a condition for indeterminacy that is entirely expressed in terms of the moment transform (and its extension to imaginary lines) of the measure, which under an additional assumption yields a converse to Carleman's criterion.  In order to state this criterion we need to introduce some notation.

Let $\mathrm{C}_+^2(I)$ denote the set of twice differentiable functions on an interval $I \subseteq \R$ whose second derivative is continuous and strictly positive on $I$. We define the set of \emph{asymptotically parabolic functions}, a notion which traces its origins to \cite{balkema:1993,balkema:1995}, as
\begin{equation}
\label{eq:definition-ap}
\ap = \left\lbrace G \in \mathrm{C}_+^2((a,\infty)), \ a \geq -\infty; \ G''\left(u+ w(G''(u))^{-\frac{1}{2}}\right) \stackrel{\infty}{\sim} G''(u), \text{ locally uniformly in } w \in \R\right\rbrace,
\end{equation}
where $f(u) \stackrel{\infty}{\sim} g(u)$ means that $\lim\limits_{u \to \infty} \frac{f(u)}{g(u)} = 1$. We are now ready to state our Tauberian type criterion for the Stieltjes moment problem.

\begin{theorem}
\label{thm:(in)determinacy-moment-sequence}
Let $\M_\nu$ be the Stieltjes moment sequence of a positive measure $\nu$ and assume that the following two conditions hold.
\begin{enumerate}[label=(\alph*)]
\item \label{item-a:thm:(in)determinacy-moment-sequence} There exists $G \in \ap$ such that
\begin{equation}
\M_\nu(n) \stackrel{\infty}{\sim} e^{G(n)}.
\end{equation}
\item \label{item-b:thm:(in)determinacy-moment-sequence} There exists $n_0 \in [0,\infty)$ such that for $n \geq n_0$, writing $\eta^2(n) = \left( \log \M_\nu(n)\right)''$, the functions
\begin{equation}
\label{eq:analytic condition on Mellin transform}
y \mapsto \left|\frac{\M_\nu\left(n+i\tfrac{y}{\eta(n)}\right)}{\M_\nu(n)}\right|
\end{equation}
are uniformly (in $n$) dominated by a function in $\Leb^1(\R)$.
\end{enumerate}
\begin{enumerate}
\item \label{item-1:thm:(in)determinacy-moment-sequence} Then, the condition
\begin{equation}
\label{eq:G* Krein integral}
\int^\infty (uG'(u)-G(u)) G''(u) e^{-\frac{G'(u)}{2}}  du < \infty \implies \nu \text{ is moment indeterminate}
\end{equation}
where, here and below, $\int^\infty$  stands for the  integral $\int_A^\infty$ for some $A>0$.
\item \label{item-2:thm:(in)determinacy-moment-sequence} If in addition $\lim\limits_{u \to \infty} ue^{-\frac{G'(u)}{2}} < \infty$ then
\begin{align}
\nu \text{ is moment determinate }
& \iff \int^\infty (uG'(u)-G(u)) G''(u) e^{-\frac{G'(u)}{2}} du = \infty  \\
& \iff \int^\infty e^{-\frac{G'(u)}{2}} du = \infty \\
\label{eq:Carleman's criterion}
& \iff \sum^\infty \M_\nu^{-\frac{1}{2n}}(n) = \infty.
\end{align}
\end{enumerate}
\end{theorem}

This theorem is proved in \Cref{subsec:proof-thm:(in)determinacy-moment-sequence}. We call it a Tauberian type result since assumptions on the moment transform alone give sufficient information regarding the measure for concluding indeterminacy. In \Cref{subsec:log-levy} below we shall provide an application of this criterion to the time-dependent moment problem for the law of log-L\'evy processes. Invoking now a useful result from Berg and Dur\'an \cite[Lemma 2.2 and Remark 2.3]{berg:2004} regarding factorization of moment sequences in relation to the moment problem, we deduce the following corollary of \Cref{thm:(in)determinacy-moment-sequence}\ref{item-1:thm:(in)determinacy-moment-sequence}.
\begin{corollary}
\label{cor:indeterminacy-moment-sequence-factorization}
Let  $\M_\mathcal{V}$ be the Stieltjes moment sequence of a positive measure $\mathcal{V}$ and suppose that, for $n \geq 0$,
\begin{equation*}
\M_\mathcal{V}(n) = \M_\nu(n) \hspace{1pt} \mathfrak{m}(n),
\end{equation*}
where $\M_\nu$ is a Stieltjes moment sequence that satisfies the assumptions of \Cref{thm:(in)determinacy-moment-sequence}, and $(\mathfrak{m}(n))_{n \geq 0}$ is a non-vanishing (i.e.~$\mathfrak{m}(n) \neq 0$ for all $n$) Stieltjes moment sequence. Then,
\begin{equation*}
\int^\infty (uG'(u)-G(u)) G''(u) e^{-\frac{G'(u)}{2}} du < \infty \implies \mathcal{V} \text{ is moment indeterminate.}
\end{equation*}
\end{corollary}
% This Corollary is proved in \Cref{subsec:proof-cor:Mv}.
We proceed by offering a few remarks regarding our criterion in relation to the recent literature on the moment problem. In \Cref{thm:(in)determinacy-moment-sequence}\ref{item-1:thm:(in)determinacy-moment-sequence} we provide a checkable criterion for indeterminacy based solely on properties of the moment transform, which seems to be new in the context of the moment problem. For example, the assumption that $\M_\nu(n) \geq cn^{(2+\epsilon)n}$ for some constants $c, \epsilon > 0$ and $n$ large enough, together with
\begin{equation}
\label{eq:condition L}
\frac{x\nu'(x)}{\nu(x)} \text{ decreases to } - \infty \text{ as } x \to \infty,
\end{equation}
where $\nu(dx) = \nu(x)dx$, allows one to conclude indeterminacy, see Theorem 5 in the nice survey \cite{lin:2017a}. The condition expressed by \eqref{eq:condition L}, which goes back to \cite{lin:1997}, is called Lin's condition in the literature, and involves the a priori assumption of the existence and differentiability of the density on a neighborhood of infinity.%, which we avoid.

In the same spirit, the integrability condition in \Cref{thm:(in)determinacy-moment-sequence}\ref{item-b:thm:(in)determinacy-moment-sequence} can be replaced by (but is not equivalent to) the assumption that $\nu(dx) = \nu(x)dx$ is such that
\begin{equation}
\label{eq:log-concavity condition}
x \mapsto -\log \nu(e^x) \text{ is convex} \text{ for } x \text{ large enough}.
\end{equation}
Under the assumption in \eqref{eq:log-concavity condition}, Pakes proved in \cite{pakes:2001} that Carleman's criterion becomes an equivalent condition for moment determinacy. However, as with Lin's condition, this involves assumptions on both the moment sequence and the density, and is a rather strong geometric requirement on the density itself. We point out that, as by-product of \Cref{thm:(in)determinacy-moment-sequence}, we have $\nu(dx) = \nu(x)dx, x> 0$, and that $\nu(e^x)$ satisfies a less stringent asymptotic condition, which is in fact implied by \eqref{eq:log-concavity condition}, see~\cite[Theorem 2.2 and Equation (4.5)]{balkema:1995}.

% Need to mention Lin and Pakes here more, to state that they also investigated under which conditions Carleman becomes iff.
In \Cref{thm:(in)determinacy-moment-sequence}\ref{item-2:thm:(in)determinacy-moment-sequence} we are able to show, under a further mild assumption on $G$, that Carleman's criterion becomes necessary and sufficient for determinacy. The additional assumption on $G$ is what allows us to connect the condition in \eqref{eq:G* Krein integral} to the finiteness of the sum in \eqref{eq:Carleman's criterion}, which is the harder of the two implications to prove. While both Lin's condition in \eqref{eq:condition L} and Pakes' condition \eqref{eq:log-concavity condition} yield converses to Carleman's criterion, we avoid having to make distinct assumptions on the moment transform and the density.

\subsection{An Abelian type tail condition for determinacy, and a converse for Krein's criterion}

The celebrated Krein's criterion, refined by Pedersen \cite{pedersen:1998}, states that if $\nu(dx) = \nu(x)dx$ and, for some $x_0 \geq 0$,
\begin{equation*}
\int_{x_0}^\infty \frac{-\log \nu(x^2)}{1+x^2}dx < \infty
\end{equation*}
then $\nu$ is moment indeterminate (the case $x_0 = 0$ yields the original version of Krein's theorem). It is also well-known that this condition is not necessary for moment indeterminacy, see the counterexample in~\cite{pedersen:1998}. In this section we provide conditions for moment determinacy that are stated in terms of the measure directly, which under some additional assumptions yields a converse to the refined Krein's criterion.

To state our result we define the set of \emph{admissible} asymptotically parabolic functions as
\begin{equation*}
\apm = \left\lbrace G_* \in \ap; \lim_{x \to \infty} \frac{G_*(x)}{x} = \infty \right\rbrace
\end{equation*}
and note that not all asymptotically parabolic functions are admissible, see e.g.~the last row of Table \ref{tab:asymptotic parabolic} in \Cref{subsec:generalization-Hardy}. The admissibility condition is equivalent to the condition that, for large enough $x$, the function $x \mapsto e^{G_*(\log x)}$ grows faster than any polynomial. Hence our reason for assuming admissibility is to avoid trivial situations in terms of the moment problem. We suggestively write $G_*$ as it will turn out that $G_*$ will be the Legendre transform of a function $G \in \ap$, see the beginning of \Cref{sec:core-proofs} and in particular \Cref{lem:admissible G closed under Legendre transform} for further details.

Next, we write, for suitable functions $f$ and $g$, $f(x) \stackrel{\infty}{=} \bigO(g(x))$ if $\limsup\limits_{x \to \infty} \left| \frac{f(x)}{g(x)} \right| < \infty$ and $f(x) \stackrel{\infty}{\asymp} g(x)$ if $f \stackrel{\infty}{=} \bigO(g(x))$ and $g(x) \stackrel{\infty}{=} \bigO(f(x))$. We also write $\overline{\nu}(x)=\int_x^{\infty}\nu(dx)$ for the tail of a probability measure $\nu$. The following result may be thought of as the Abelian counterpart to the Tauberian type result in \Cref{thm:(in)determinacy-moment-sequence}.

% To do:
% - incorporate Carleman type condition and same formatting of theorem 2.1
% - improve remark 1, especially sub-linear scale comment
% - expound on Hardy criterion remark, make it clear that sharpness of result did not seem to be known until recently
% - update and clarify proofs
% - improve table and make it clear what is being displayed, also that further examples can be provided in Bernstein-Urbanik section

\begin{theorem}
\label{thm:(in)determinacy-density}
Let $\nu$ be a probability measure with all positive moments finite.
\begin{enumerate}
\item \label{item-1:thm:(in)determinacy-density} Suppose that there exists $G_* \in \apm$ such that either
\begin{equation}
\label{eq:big-O tail}
\overline{\nu}(x) \stackrel{\infty}{=} \bigO(e^{-G_*(\log x)}),
\end{equation}
or, if $\nu(dx) = \nu(x)dx$, that
\begin{equation}
\label{eq:big-O density}
\nu(x) \stackrel{\infty}{=} \bigO(e^{-G_*(\log x)}).
\end{equation}
 Then, writing $\overleftarrow{G_*'}$ for the inverse of the continuous and increasing function $G_*'$,
\begin{equation}
\label{eq:determinate sum}
\sum^\infty e^{-\frac{\overleftarrow{G_*'}(n)}{2}} = \infty \implies \nu \text{ is moment determinate}.
\end{equation}
\item \label{item-2:thm:(in)determinacy-density} If in addition
\begin{equation}
\nu(x) \stackrel{\infty}{\asymp} e^{-G_*(\log x)}
\end{equation}
 and $\lim\limits_{x \to \infty} G_*'(x)e^{-\frac{x}{2}} < \infty$, then
\begin{align}
\label{eq:equivalent sum}
\nu \text{ is moment indeterminate}
& \iff \sum^\infty e^{-\frac{\overleftarrow{G_*'}(n)}{2}} < \infty \\
& \iff \int^\infty G_*(x)e^{-\frac{x}{2}}dx < \infty \\
& \iff \int^\infty \frac{-\log\nu(x^2)}{1+x^2} dx < \infty.
\end{align}
\end{enumerate}
\end{theorem}
This theorem is proved in \Cref{subsec:proof-thm:(in)determinacy-density}. It leads to a generalization of Hardy's condition for moment determinacy, which was proved by Hardy in a series of papers \cite{hardy:1917,hardy:1918} and seemed to have gone unnoticed in the probabilistic/moment problem literature until the recent exposition by Stoyanov and  Lin \cite{stoyanov:2013}, see also \cite{lin:2017a}. The criterion states that if
\begin{equation}
\label{eq:Hardy's condition}
\int_0^\infty e^{c\sqrt{x}}\nu(dx) < \infty \text{ for some } c > 0
\end{equation}
then $\nu$ is determinate.
% rephrase in terms of nu, otherwise this looks weird

\begin{corollary}
\label{cor:generalized-Hardy}
Let $\nu$ be a probability measure with all positive moments finite.
\begin{enumerate}
\item Hardy's condition is satisfied, i.e.~\eqref{eq:Hardy's condition} holds, if and only if
\begin{equation*}
\overline{\nu}(x) \stackrel{\infty}{=} \bigO(e^{-c\sqrt{x}}).
\end{equation*}
Consequently Hardy's condition implies that \eqref{eq:big-O tail} and \eqref{eq:determinate sum} of \Cref{thm:(in)determinacy-density}\ref{item-1:thm:(in)determinacy-density} are satisfied, with $x \mapsto G_*(x) = ce^{\frac{x}{2}} \in \apm$.
\item \label{item-2:cor:generalized-Hardy} If for some $\alpha>0$
\begin{equation}
\label{eq:big O Hardy example}
\nu(x) \stackrel{\infty}{\asymp} e^{-\frac{\alpha\sqrt{x}}{\log x}}
\end{equation}
then $\nu$ does not satisfy Hardy's criterion, i.e.~
\begin{equation*}
\int_0^\infty e^{c\sqrt{x}}\nu(x)dx = \infty \quad \forall c > 0,
\end{equation*}
yet $\nu$ is moment determinate.
\end{enumerate}
\end{corollary}
This corollary is proved in \Cref{subsec:proof-cor:generalized-Hardy}. Now, we proceed with some remarks on the two last results. The fact that \Cref{thm:(in)determinacy-density}\ref{item-1:thm:(in)determinacy-density} leads to a generalization of Hardy's condition shows that the assumptions we make are rather weak yet still yield the moment determinacy of $\nu$. Note that the requirement in \eqref{eq:big-O tail} or in \eqref{eq:big-O density} does not trivially imply moment determinacy since a function $G_* \in \apm$ may be sublinear at the log-scale, e.g.~$G_*(\log x)=x^\alpha$ for $\alpha > 0$.

It was shown in Stoyanov and Lin \cite{stoyanov:2013} that Hardy's condition implies Carleman's criterion, so that the same argument that disproves the necessity of Carleman's criterion also shows that Hardy's condition is not necessary for moment determinacy. This argument, which goes back to Heyde \cite{heyde:1963}, involves the subtle manipulation of point mass at the origin. In \Cref{cor:generalized-Hardy}\ref{item-2:cor:generalized-Hardy} we are able to give explicit examples of densities, characterized only by their large asymptotic behavior, for which Hardy's condition fails yet, by \Cref{thm:(in)determinacy-density}\ref{item-2:thm:(in)determinacy-density}, Carleman's criterion holds.

In \Cref{thm:(in)determinacy-density}\ref{item-2:thm:(in)determinacy-density} we give necessary and sufficient conditions on the density $\nu$ for moment indeterminacy, and also show that Krein's criterion becomes necessary and sufficient in our context. The existing criteria in the literature that give converses to Krein's theorem require either the differentiability of the density, such as Lin's condition in \eqref{eq:condition L}, or an exact representation for the density, e.g.~\cite[Theorem 4]{pakes:2001}, neither of which we suppose.

% Need references for converse statement not being true, and also a way to connect to the Acta paper of Berg and Thill
%\begin{comment}
%\begin{remark}
%\label{rem:stationary-excess}
%The proof of \Cref{thm:(in)determinacy-density}\ref{item-1:thm:(in)determinacy-density} relies on the fact that the moment determinacy of the probability density
%\begin{equation}
%\overline{\nu}(x) = \frac{\overline{\nu}(x,\infty)}{\M_\nu(1)}, x>0,
%\end{equation}
%implies the moment determinacy of $\nu$. Here $\overline{\nu}$ is the density of the so-called stationary-excess distribution (of order 1) of $\nu$. In general, the converse is not true, that is the indeterminacy of $\overline{\nu}$ does not imply the indeterminacy of $\nu$, see~\cite{berg:2004a} for an example.
%\end{remark}
%\end{comment}

Finally, we mention that, in  \cite{patie:2018}, we apply \Cref{thm:(in)determinacy-density} to study the log-L\'evy moment problem for the so-called Berg-Urbanik semigroups.

\section{Applications} \label{sec:applications}

\subsection{The log-L\'evy moment problem} \label{subsec:log-levy}

% distribution. It is well-known that the log-normal distribution is indeterminate by its moments and
%one can easily show, by applying a refinement of Krein's theorem due to Pedersen \cite{pedersen:1998} to the corresponding densities, that the geometric Brownian motion $(e^{B_t})_{t \geq 0}$ is indeterminate by its moments for all $t > 0$. Consequently, a result by Berg and Dur\'an \cite[Lemma 2.2]{berg:2004} implies that $(e^{Y_t})_{t \geq 0}$ is moment indeterminate for all $t > 0$, whenever $Y_t = B_t + Z_t$ with $(Z_t)_{t \geq 0}$ a L\'evy process admitting exponential moments. In the absence of a diffusive component, it is possible for $e^{Y_t}$ to be moment determinate for $0 \leq t < \scr{T} < \infty$, and then indeterminate for $t > \scr{T}$. It is also possible for $e^{Y_t}$ to be moment determinate for all $t \geq 0$. In this paper we study a class of multiplicative convolution semigroups -- the Berg-Urbanik semigroups -- for which this kind of rich and interesting behavior of the log-L\'evy moment problem can be observed

% Note where the log-Levy moment problem has been studied before

One of the most famous indeterminate measures is the log-normal distribution, and the indeterminacy of this measure has the consequence that the random variable $e^{B_t}$ is moment indeterminate for all $t > 0$, where $B = (B_t)_{t \geq 0}$ is a standard Brownian motion. In this section we apply \Cref{thm:(in)determinacy-moment-sequence} to study this time-dependent moment problem when $B$ is replaced by a L\'evy process (admitting all exponential moments), which we call the log-L\'evy moment problem.

We recall that a (one-dimensional) L\'evy process $Y = (Y_t)_{t \geq 0}$ is a $\R$-valued stochastic process with stationary and independent increments, that is continuous in probability, and such that $Y_0 = 0$ almost surely (a.s.). Such processes are fully characterized by the law of $Y_1$, which is known to be infinitely divisible and whose characteristic exponent is given by
\begin{equation}
\label{eq:LK-representation}
\Psi(u) = bu + \frac{1}{2}\sigma^2u^2 + \int_{-\infty}^\infty \left(e^{u r} - 1 - u r \mathbb{I}_{\{|r| \leq 1\}}\right) \Pi(dr), \quad u \in i\R,
\end{equation}
with $b \in \R$, $\sigma \geq 0$, and $\Pi$ a $\sigma$-finite, positive measure satisfying $\Pi(\{0\}) = 0$ and the integrability condition $
\int_{-\infty}^\infty \min(1,r^2) \ \Pi(dr) < \infty.$ As we are interested in the log-L\'evy moment problem we only consider L\'evy processes admitting all positive exponential moments, i.e.~$\E[e^{uY_t}] < \infty$ for all $u, t \geq 0$. This condition is equivalent to $\Psi$ admitting an analytical extension to the right-half plane, still denoted by $\Psi$, which in terms of the L\'evy measure can be expressed as
\begin{equation*}
\int_1^\infty e^{ur} \Pi(dr) < \infty, \quad  u\geq 0,
\end{equation*}
see~\cite[Theorem 25.3 and Lemma 25.7]{sato:2013}. In this case we have that
\begin{equation*}
\E[e^{uY_t}] = e^{t\Psi(u)}, \quad u \geq 0.
\end{equation*}

%In this section we apply \Cref{thm:(in)determinacy-moment-sequence} to investigate the time-dependent moment problem for stochastic processes of the form $(X_t)_{t \geq 0} = (e^{Y_t})_{t \geq 0}$, where $(Y_t)_{t \geq 0}$ is a one-dimensional L\'evy process such that $\E[e^{nY_t}] < \infty$ for all $n, t \geq 0$. We call this the log-L\'evy moment problem. A simple example is given by the geometric Brownian motion $(e^{B_t})_{t \geq 0}$, which is log-normally distributed and is therefore well-known to be indeterminate for all $t > 0$.
%
%Write $\mc{N}$ for the set of functions $\Psi$ having the form
%\begin{equation}
%\Psi(u) = bu + \frac{\sigma^2}{2}u^2 + \int_{-\infty}^\infty (e^{ux}-1-ux)\Pi(dx),
%\end{equation}
%where $b \in \R$, $\sigma \geq 0$, and $\Pi$ is a $\sigma$-finite, positive measure satisfying the integrability conditions
%\begin{equation}
%\int_{-\infty}^\infty (|x| \wedge x^2) \Pi(dx) < \infty, \quad \text{and} \quad \int_1^\infty e^{ux} \Pi(dx) < \infty, \quad \text{for } u \geq 0.
%\end{equation}
%The set $\mc{N}$ is a subset of the Laplace exponents of L\'evy processes, i.e.~for $t, u \geq 0$,
%\begin{equation}
%\E[e^{uY_t}] = e^{t\Psi(u)},
%\end{equation}
%and the integrability conditions on the L\'evy measure $\Pi$ ensure this is well-defined.

\begin{theorem} Let $Y = (Y_t)_{t \geq 0}$ be a L\'evy process admitting all exponential moments.
\label{thm:threshold-log-Levy}
\begin{enumerate}
\item \label{item-1:thm:threshold-log-Levy} If in \eqref{eq:LK-representation} $\sigma^2 > 0$, then the random variable $e^{Y_t}$ is moment indeterminate for any $t > 0$.
\item \label{item-2:thm:threshold-log-Levy} If $\Psi(u) = u\log(u+1)$, $u \geq 0$, then the random variable $e^{Y_t}$ is moment determinate if and only if $t \leq 2$.
\end{enumerate}
\end{theorem}

This Theorem is proved in \Cref{subsec:proof-thm:threshold-log-Levy}. In \Cref{thm:threshold-log-Levy}\ref{item-2:thm:threshold-log-Levy} we provide an example of a L\'evy exponent such that the log-L\'evy moment problem is determinate up to a threshold time, and then indeterminate afterwards. This phenomenon has been observed in the literature by Berg in \cite{berg:2005} for the so-called Urbanik semigroup and in \cite{patie:2018} we extend Berg's result to a large class of multiplicative convolution semigroups, which do not have a log-normal component. We also mention that we prove \Cref{thm:threshold-log-Levy}\ref{item-2:thm:threshold-log-Levy} also via an application of \Cref{thm:(in)determinacy-moment-sequence} and interestingly, the additional condition in \Cref{item-2:thm:(in)determinacy-moment-sequence} of \Cref{thm:(in)determinacy-moment-sequence} is only fulfilled for $t \geq 2$.

We point out that, as a by-product of the proof of \Cref{thm:threshold-log-Levy}\ref{item-1:thm:threshold-log-Levy}, in the case of spectrally-negative L\'evy processes, we get the following large asymptotic  behavior of their densities, valid for all $t>0$, which seems to be new in the L\'evy literature. Note that, from \cite[Ex.~29.14, p.~194]{sato:2013}, we have, for $\sigma^2 > 0$ and any fixed $t > 0$, $\mathbb{P}(Y_t \in dy) = f_t(y)dy, y\in \R$, where $y \mapsto f_t(y) \in \textrm{C}^{\infty}(\R)$, all derivatives of which tend to $0$ as $|y| \to \infty$, and where $\textrm{C}^{\infty}(\R)$ stands for the space of infinitely differentiable functions on $\R$.
\begin{corollary}
\label{cor:asymptotic-spectrally-negative}
Assume that, in \eqref{eq:LK-representation}, $\sigma^2 > 0$ and $\overline{\Pi}(0)= \Pi(0,\infty) = 0$. Then, for any fixed $t > 0$, we have the following large asymptotic behavior of the density
\begin{equation}
\label{eq:asymptotic-spectrally-negative}
f_t(t\Psi'(y)) \stackrel{\infty}{\sim}  \frac{1}{\sqrt{2\pi \sigma^2 t}} e^{-\frac{1}{2}t\sigma^2y^2+ty^2\int_{-\infty}^0 e^{yr}r\Pi(-\infty,r)dr}.
\end{equation}
% \stackrel{\infty}{\sim}\frac{1}{\sqrt{2\pi \sigma^2 t}}e^{-\frac{y^2}{2t\sigma^2}\left(1-\frac{2}{\sigma^2}\int_{-\infty}^0 \exp\left(\frac{yr}{t\sigma^2}\right)r\Pi(-\infty,r)dr\right)}.
\end{corollary}
The proof of this corollary is given in \Cref{subsec:proof-cor:asymptotic-spectrally-negative}. When $\Pi \equiv 0$ then one can easily invert $\Psi'$ in \eqref{eq:asymptotic-spectrally-negative} to reveal the classical asymptotic for the density of a Brownian motion with drift. We point out that if $\Pi(dr)=\alpha |r|^{-\alpha-1}dr, r<0, 0<\alpha <2,$ that is the L\'evy measure of a spectrally-negative $\alpha$-stable L\'evy process, then $\int_{-\infty}^0 e^{yr} r\Pi(-\infty,r)dr=-\Gamma(2-\alpha)y^{\alpha-2}$. As an illustration, when  $\alpha=\frac{3}{2}$ and we choose $\Psi(u) = \frac{1}{2}\sigma^2 u^2 + \frac{2}{3}u^{\frac{3}{2}}$, a straightforward computation allows one to get that, for $t > 0$,
\begin{equation*}
f_t(y)\stackrel{\infty}{\sim}  \frac{1}{\sqrt{2\pi \sigma^2 t}} e^{-\frac{y^2}{2\sigma^2 t} -H(y,t)},
\end{equation*}
where $H$ is given by
\begin{equation*}
H(y,t) = \frac{\left(y+\frac{1}{2\sigma^2}\right)\sqrt{y+\frac{1}{4\sigma^2}}}{\sigma^3 t} + \frac{\left(y + \frac{1}{\sigma^2} - \sqrt{\frac{y}{\sigma^2}+\frac{1}{4\sigma^2}}\right)^\frac{3}{2}}{3\sigma^3 \sqrt{t}}.
\end{equation*}
Note that, for fixed $t$,
\begin{equation*}
H(y,t) \stackrel{\infty}{\sim} \frac{y^\frac{3}{2}}{\sigma^3 t}\left(1+\frac{\sqrt{t}}{3}\right)
\end{equation*}
and that the two terms in the above asymptotic scale differently in $t$.

% \left(\frac{y}{\sigma^2 t}\right)^{\frac{3}{2}}\left(\sqrt{t} + \frac{t}{3}\right),

%
%A generalization of this example is given in the following result.
%
%\begin{theorem}
%\label{thm:indeterminacy-diffusion-component}
%For $\sigma > 0$, let $Y_t = \sigma B_t + Z_t$, where $(B_t)_{t \geq 0}$ is a standard Brownian motion and $(Z_t)_{t \geq 0}$ is a L\'evy process with $\E[e^{nZ_t}] < \infty$ for all $n, t \geq 0$. Then $e^{Y_t}$ is moment indeterminate for all $t > 0$.
%\end{theorem}
%
%\noindent \emph{This Theorem is proved in \Cref{subsec:proofs-log-Levy}}
%
%\begin{remark}
%The proof essentially follows from a lemma by Berg and Dur\'an, which states that if $\mu$ is an indeterminate measure on $[0,\infty)$ and $\nu$ is not identically a point mass at zero, then $\mu \diamond \nu$ is indeterminate, where $\diamond$ denotes product convolution, see~\cite[Lemma 2.2 and Remark 2.3]{berg:2004}. Despite the simplicity of the proof, we are not aware of a result similar to \Cref{thm:indeterminacy-diffusion-component} having been stated in the literature, and hence include it here.
%\end{remark}

\subsection{Some new and classical examples of asymptotic behavior for densities} \label{subsec:generalization-Hardy}

In the following table we list some further examples of functions $G_* \in \ap$, and state whether or not any probability density $\nu$ satisfying
\begin{equation*}
\nu(x) \stackrel{\infty}{\asymp} e^{-G_*(\log x)}
\end{equation*}
admits all moments, and if so, whether it is moment determinate, possibly as a function of some parameter.

\begingroup
\renewcommand{\arraystretch}{1.5}
\begin{table}[h!]
\centering
\begin{tabular}{llll}
$\nu(x) \stackrel{\infty}{\asymp} e^{-G_*(\log x)}$ & $\M_\nu(n) <+\infty$ & parameter & moment (in)determinacy \\ \hline
$\exp\left(-\frac{\alpha\sqrt{x}}{\log x}\right)$ & $n \in \N$ & $\alpha > 0$ & determinate $\forall \alpha > 0$ \rule{0pt}{3ex} \\
$\exp\left(-x^\beta\right)$ & $n \in \N$ & $\beta > 0$ & determinate $\iff \beta \geq \frac{1}{2}$ \\
$\exp\left(-(\log x)^\delta\right)$ & $n \in \N$ & $\delta > 1$ & indeterminate $\forall \delta > 1$ \\
$\exp\left(-\kappa(\log x) \log (\log x)\right)$ & $n \in \N$ & $\kappa > 0$ & indeterminate $ \forall \kappa > 0$ \\
$\exp\left(-(\log x)^\lambda + \log x\right)$ & $n \leq 1$ &$\lambda \in (0,1)$ & \\
\end{tabular}
\vspace{10pt}
\caption{Examples of asymptotically parabolic functions and moment (in)determinacy of $\nu$.} \label{tab:asymptotic parabolic}
\end{table}
\endgroup

The first row corresponds to \eqref{eq:big O Hardy example} of \Cref{cor:generalized-Hardy}. The example from the second row of Table \ref{tab:asymptotic parabolic} is well-known in the literature. The authors in \cite{stoyanov:2013} use it to illustrate that the exponent 1/2 (i.e.~square root) in Hardy's condition cannot be improved, and in this sense Hardy's condition is the optimal version of Cramer's condition for moment determinacy. As can be readily checked, the function $x \mapsto e^{\beta x} \in \apm$, for all $\beta > 0$, and the condition in \eqref{eq:determinate sum} of \Cref{thm:(in)determinacy-density}\ref{item-2:thm:(in)determinacy-density} organically reveals the threshold value of $\beta = \frac{1}{2}$.

% I could make this a result on its own, perhaps a proposition or corollary.

This example also illustrates how a natural transformation of functions $G_*\in \ap$ influences the moment determinacy of $\nu(x) \stackrel{\infty}{\asymp} e^{-G_*(\log x)}$. For $c > 0$ let $\operatorname{d}_c$ denote the dilation operator, acting on functions $f:\R_+ \to \R$ via
\begin{equation*}
\operatorname{d}_c f (x) = f(cx).
\end{equation*}
From the fact that $\ap$ is a convex cone, we get that, for any $c > 0$ and $G_* \in \ap$, $cG_* \in \ap$. However, we also have, for any $c > 0$ and $G_* \in \ap$, that $\operatorname{d}_c G_* \in \ap$ since
\begin{equation*}
\left(\operatorname{d}_c G_*\right)''(x) = c^2 G_*''(cx)
\end{equation*}
and hence the defining properties of $G_*$ in \eqref{eq:definition-ap} carry over to $\operatorname{d}_c G_*$. Now let $G_*(x) = e^{\beta x}$ and consider $\nu(x) \stackrel{\infty}{\asymp} \exp(-x^\beta)$. Then taking $cG_*$ leads to $\nu(x) \stackrel{\infty}{\asymp} \exp(-cx^\beta)$, which is moment determinate if and only if $\beta \geq \frac{1}{2}$, independently of $c > 0$, while taking $\operatorname{d}_c G_*$ leads to $\nu(x) \stackrel{\infty}{\asymp} \exp(-x^{c\beta})$, which is moment determinate if and only if $c \beta \geq \frac{1}{2}$.

\section{Proofs} \label{sec:core-proofs}

% Need to introduce self-neglecting here, since I moved it away from main results
% Need to update proofs according to notation

Before we begin with the proof of the main results we introduce some notation and  state and prove some preliminary lemmas that will be useful below. For  $a \geq -\infty$,  we say that a function $s:(a,\infty) \to (0,\infty)$ is  \emph{self-neglecting} if
\begin{equation*}
\lim\limits_{u \to \infty} \frac{s(u+ws(u))}{s(u)} = 1 \quad \text{ locally uniformly in } w \in \R.
\end{equation*}
Hence a function $G \in \ap$ if and only if $G \in \mathrm{C}_+^2((a,\infty))$ and its \emph{scale function} $s_G(u) = (G''(u))^{-\frac{1}{2}}$ is self-neglecting. Note that the self-neglecting property is closed under asymptotic equivalence, that is if $s(u) \stackrel{\infty}{\sim} p(u)$ and $s$ is self-neglecting then $p$ is self-neglecting. We refer to \cite[Section 2.11]{bingham:1989} for further information regarding self-neglecting functions.
Next, a function $b$ is said to be \emph{flat} with respect to $G \in \ap$ if
\begin{equation*}
\lim\limits_{u \to \infty} \frac{b(u+ws_G(u))}{b(u)} = 1 \quad \text{locally uniformly in } w \in \R,
\end{equation*}
where $s_G$ is the scale function of $G$. It is immediate from the definition that both $s_G$ and $1/s_G$ are flat with respect to $G$. In the following lemma we collect some results regarding self-neglecting and flat functions. They are essentially known in the literature, however, for sake of completeness,  we provide  their proofs. For two functions $f$ and $g$ we write $f(u)\stackrel{\infty}{=} \littleo(g(u))$ if $\lim\limits_{u \to \infty} \left|\frac{f(u)}{g(u)}\right| = 0$.
\begin{lemma} $ $
\label{lem:self-neglecting-flat}
\begin{enumerate}
\item \label{item-1:lem:self-neglecting-flat} Let $s:(a,\infty) \to (0,\infty)$, $a \geq -\infty$, be self-neglecting. Then $s(u) \stackrel{\infty}{=} \littleo(u)$.
\item \label{item-2:lem:self-neglecting-flat} Let $b$ be flat with respect to $G \in \ap$. Then,
\begin{equation*}
\lim\limits_{u \to \infty} \frac{\log b(u)}{G(u)} = 0.
\end{equation*}
\end{enumerate}
\end{lemma}
\begin{proof}
The representation theorem for self-neglecting functions, see~\cite[Theorem 2.11.3]{bingham:1989}, states that
\begin{equation}
\label{eq:s-n-representation}
s(u) = c(u)\int_0^u g(y)dy,
\end{equation}
where $c$ is measurable with $\lim\limits_{u \to \infty} c(u) = \gamma \in (0,\infty)$ and $g \in \mathrm{C}^\infty(\R)$  with $\lim\limits_{y\to \infty} g(y) = 0$. Now let $K > 0$ be fixed. From \eqref{eq:s-n-representation} we get
\begin{equation*}
\frac{s(u)}{uc(u)} =  \frac{1}{u}\int_0^u g(y) dy \leq \frac{K}{u} \sup_{y \in [0,K]} |g(y)| + \frac{u-K}{u} \sup_{y \in [K,u]} |g(y)|
\end{equation*}
and thus we deduce that
\begin{equation*}
\lim\limits_{u \to \infty} \frac{s(u)}{u} \leq \gamma \sup_{y \geq K} |g(y)|.
\end{equation*}
Since $K > 0$ was arbitrary, $s$ is positive and $\lim\limits_{y \to \infty} g(y) = 0$, the conclusion of \Cref{lem:self-neglecting-flat}\ref{item-1:lem:self-neglecting-flat} then follows. For the next claim, we invoke \cite[Proposition 3.2]{balkema:1993} to get that, on the one hand, there exists a $\mathrm{C}^\infty(\R)$-function $\beta$ such that $b(u) \stackrel{\infty}{\sim} \beta(u)$ and
\begin{equation*}
\lim\limits_{u \to \infty} \frac{s_G(u) \beta'(u)}{\beta(u)} = 0.
\end{equation*}
On the other hand, Proposition 5.8 from the aforementioned  paper yields
\begin{equation*}
\lim\limits_{u \to \infty} \frac{1}{s_G(u)G'(u)}=0.
\end{equation*}
Then, by L'Hospital's rule,
\begin{equation*}
\lim\limits_{u \to \infty} \frac{\log \beta(u)}{G(u)} = \lim\limits_{u \to \infty} \frac{\beta'(u)}{\beta(u)G'(u)} = \lim\limits_{u \to \infty} \frac{s_G(u)\beta'(u)}{\beta(u)} \frac{1}{s_G(u)G'(u)} = 0
\end{equation*}
and the claim follows from $b(u) \stackrel{\infty}{\sim} \beta(u)$ together with $\lim\limits_{u \to \infty} G(u) = \infty$, see \cite[Theorem 5.4]{balkema:1993}.
%Next, if $b$ is flat with respect to $G \in \ap$ then, by \cite[Proposition 3.2]{balkema:1993}, there exists $f \in \mathrm{C}^\infty(\R)$ such that $b(x) \stackrel{\infty}{\sim} f(x)$ and
%\begin{equation*}
%\lim_{x \to \infty} \frac{s_G(x) f'(x)}{f(x)} = 0.
%\end{equation*}
%Furthermore, from \cite[Proposition 5.8]{balkema:1993} we have that $\lim_{x \to \infty} s_G(x) G'(x) = \infty$. An application of L'H\^{o}pital's rule, combined with these two properties, yields
%\begin{equation*}
%\lim_{x \to \infty} \frac{\log f(x)}{G(x)} = \lim_{x \to \infty} \frac{f'(x)}{f(x)G'(x)} = \lim_{x \to \infty} \frac{s_G(x)f'(x)}{f(x)} \frac{1}{s_G(x)G'(x)} = 0.
%\end{equation*}
%This concludes the proof, since one necessarily has $\lim_{x \to \infty} G(x) = \infty$.
\end{proof}
% Involution
For the next lemma we recall that the Legendre transform $G_*$ of a convex function $G$ is defined by
\begin{equation*}
G_*(x) = \sup_{u \in \R} \{xu-G(u)\}
\end{equation*}
for all $x \in \R$ such that the right-hand side is finite.
When $G$ is differentiable and  strictly convex, which is the case when $G \in \ap$, the supremum is attained at the unique point $u = \overleftarrow{G'}(x)$, where $\overleftarrow{G'}$ stands for the inverse of the continuous and increasing function $G'$, so that the Legendre transform is given by
\begin{equation}
\label{eq:derivative-Legendre}
G_*(x) = x\overleftarrow{G'}(x) - G(\overleftarrow{G'}(x)).
\end{equation}
The function $G_*$ is always convex. The Legendre transform is an involution on the space of convex functions, i.e.~for $G$ convex one has $({G_*})_* = G$. In the next lemma we prove another closure property regarding the Legendre transform, pertaining to the set $\apm$.

\begin{lemma}
\label{lem:admissible G closed under Legendre transform}
The set of admissible asymptotically parabolic functions is closed under Legendre transform, that is if $G_* \in \apm$ then $({G_*})_* = G \in \apm$. Consequently if $G_* \in \apm $ then~$\lim\limits_{x \to \infty} G_*'(x) = \lim\limits_{u \to \infty} G'(u) = \infty$.
\end{lemma}

\begin{proof}
Let $G_* \in \apm$. Since $G_*:(a,\infty) \to \R$ for some $a \in [-\infty,\infty)$ and $G_*$ is convex, we have, for all $x,y \in (a,\infty)$, the standard inequality
\begin{equation*}
G_*(x) - G_*(y) \leq G_*'(x)(x-y).
\end{equation*}
 Fixing $y$ and letting $x \to \infty$ we see that the admissibility property implies that $\lim\limits_{x\to\infty} G_*'(x) = \infty$. In \cite{balkema:1993} it was shown that the set of asymptotically parabolic functions is closed under Legendre transform, in the sense that the restriction of $({G_*})_* = G$ to the image of $(a,\infty)$ under $G_*'$ is asymptotically parabolic. Since, the image of $(a,\infty)$ under $G_*'$ is $(b,\infty)$ for some $b \in [- \infty,\infty)$, it follows that $G$ restricted to $(b,\infty)$ is asymptotically parabolic. Hence it remains to show that $G$ is admissible. To this end, we consider the function
\begin{equation*}
f(x) = e^{-G_*(x)}\mathbb{I}_{\{x > a\}}.
\end{equation*}
The admissibility of $G_*$ implies that, for any $n \geq 1$,
\begin{equation*}
\lim_{x \to \infty} f(x)e^{nx} = \lim_{x \to \infty} e^{nx-G_*(x)}\mathbb{I}_{\{x > 0\}} = 0
\end{equation*}
i.e.~that $f$ has a Gaussian tail in the sense of \cite{balkema:1995}. This in turn yields that, for any $n \geq 1$,
\begin{equation*}
\lim\limits_{u \to \infty} e^{nu}(1-F(u)) = \lim\limits_{u \to \infty} e^{nu}\int_u^\infty f(x)dx = 0
\end{equation*}
which is equivalent to
\begin{equation}
\label{eq:1-F thin tail}
\lim\limits_{u \to \infty} e^{-nu}\int_{-\infty}^\infty e^{ux}f(x)dx = \infty,
\end{equation}
see e.g.~the discussion after \cite[Theorem C]{balkema:1995}. However, the Gaussian tail property of $f$ allows us to invoke the same result to conclude that
\begin{equation*}
\int_{-\infty}^\infty e^{ux}f(x)dx \stackrel{\infty}{\sim} \frac{\sqrt{2\pi}}{s_{G}(u)}e^{G(u)}.
\end{equation*}
This asymptotic, combined with $\log s_{G}(u) \stackrel{\infty}{=} \littleo(G(u))$ from \Cref{lem:self-neglecting-flat}\ref{item-2:lem:self-neglecting-flat} and \eqref{eq:1-F thin tail}, allows us to conclude that $G$ is admissible, from which $ \lim\limits_{u \to \infty} G'(u) = \infty$ follows as before.
\end{proof}

%
%% Don't need this lemma anymore
%
%\begin{lemma}
%\label{lem:thin tail}
%Let $\nu$ be a probability measure supported on $[0,\infty)$ admitting all positive moments and let $\mu$ be the measure defined by
%\begin{equation}
%\mu(dx) = e^x\bar{\nu}(dx),
%\end{equation}
%where $f:\R \to \R$ be the probability density given by
%\begin{equation*}
%f(x) = e^x \nu(e^x).
%\end{equation*}
%Then $\F f(u) = \int_{-\infty}^\infty e^{-ux}f(x)dx$ satisfies
%\begin{align*}
%& \F f (u) < \infty \quad \text{ for } \quad u \geq 0, \\
%& \F f(u)e^{-ku} \to \infty \quad \text{ as } \quad u \to \infty, \quad \forall k \geq 1.
%\end{align*}
%\end{lemma}
%
%\begin{proof}[Proof of \Cref{lem:thin tail}]
%By definition of $f$ and a change of variables, we have
%\begin{equation*}
%\F f(u) = \M_{\nu}(u+1).
%\end{equation*}
%Hence, for $n \in \N$,
%\begin{equation*}
%\F f(n) = \mathfrak{m}(n),
%\end{equation*}
%where $\mathfrak{m}$ is the moment function of $\nu$, which suffices to establish the first claim. To prove the second claim, let $k \geq 1$ be fixed and choose $M$ such that $\log M > k$. Then,
%\begin{equation*}
%\mathfrak{m}(n) = \int_0^\infty x^n \nu(x)dx \geq \int_M^\infty x^n \nu(dx) dx \geq \bar{\nu}(M) M^n,
%\end{equation*}
%where $\bar{\nu}$ is the tail of $\nu$ and $\bar{\nu}(M) > 0$ by assumption on the support of $\nu$. Hence, by choice of $M$,
%\begin{equation*}
%\mathfrak{m}(n)e^{-kn} \geq \nu((M,\infty	)) e^{n(\log M - k)} \to \infty \quad \text{ as } n \to \infty.
%\end{equation*}
%\end{proof}

In the  following  we provide a Tauberian result on the moment transform which is an analogue to the one obtained for the bi-lateral Laplace transform by Feigin and Yaschin, see~\cite[Theorem 3]{feigin:1983}.%, which can also be found in \cite{balkema:1995}.
\begin{proposition}
\label{prop:analogue-Tauberian-theorem}
Let $\M_\nu$ be the Stieltjes moment sequence of a positive measure $\nu$, and suppose that the conditions in \Cref{thm:(in)determinacy-moment-sequence} are satisfied. Then, $\nu$ is absolutely continuous with respect to the Lebesgue measure and its density $\nu(dx) = \nu(x)dx$, $x > 0$, satisfies
\begin{equation*}
\nu(x) \stackrel{\infty}{\sim} \frac{1}{\sqrt{2\pi}}\frac{e^{-G_*(\log x)}}{ x s_{G_*}(\log x)}  ,
\end{equation*}
where $G_*$ is the Legendre transform of $G$ and $s_{G_*}$ is its scale function. Furthermore, $G_* \in \apm$.
\end{proposition}

\begin{proof}
Let $\mu$ be the pushforward of the measure $\nu$ under the map $x \mapsto \log x$, meaning that $\mu(dy)=\nu(e^y)e^ydy,y\in \R,$ when $\nu$ is absolutely continuous with a density denoted again by $\nu$. It is immediate that $\mu$ is a probability measure with $\supp(\mu) = \R$, and for $u \geq 0$,
\begin{equation*}
\F_\mu(u) = \int_{-\infty}^\infty e^{uy}\mu(dy) = \int_0^\infty x^u \nu(dx) = \M_\nu(u),
\end{equation*}
where the left-hand equality sets a notation. Since $\nu$ admits all positive moments we get that $\F_\mu(u) < \infty$ for all $u \geq 0$. Let $k \geq 1$ be fixed and choose $M$ large enough such that $\log M > k$. Then,
\begin{equation*}
\M_\nu(n) = \int_0^\infty x^n \nu(dx) \geq \int_M^\infty x^n \nu(dx) \geq \overline{\nu}(M) M^n
\end{equation*}
and $\overline{\nu}(M) > 0$ by assumption on the support of $\nu$. By the choice of $M$,
\begin{equation*}
\lim_{n\to\infty}\M_\nu(n)e^{-kn} \geq \lim_{n\to\infty}\overline{\nu}(M) e^{n(\log M - k)} = \infty,
\end{equation*}
from which we conclude that $\lim\limits_{u \to \infty} \F_\mu(u)e^{-ku} = \infty$. Let us write $F$ for the cumulative distribution function of $\mu$. Then the properties $\F_\mu(u) < \infty$, for all $u \geq 0$, and $\lim\limits_{u \to \infty} \F_\mu(u)e^{-ku} = \infty$ for any $k \geq 1$, are equivalent to $F$ having a very thin tail in the sense of \cite{balkema:1995}. Note that, since $\M_\nu(n) \stackrel{\infty}{\sim} e^{G(n)}$, the property of $F$ having a very thin tail  which by \Cref{lem:admissible G closed under Legendre transform} and its proof  implies that both $G, G_* \in \apm$.
% Make Esscher transform E and normalized Esscher maybe E*.
% Detail computations about Esscher transform and normalization, the mean and variance, and go through the steps.
Next, for $n \geq 0$, we recall that the Esscher transform of $\mu$ is the probability measure whose cumulative distribution function is given by $\int_{-\infty}^t e^{nx} \mu(dx) / \int_{-\infty}^\infty e^{nx} \mu(dx)$, which is well-defined thanks to the fact that $F$, its distribution, has a very thin tail. Write $\mathcal{E}_\mu(n)$ for the normalized Esscher transform of $\mu$, which means that its bi-lateral Laplace transform takes the form
\begin{equation}
\label{eq:F-form-1}
\F_{\mathcal{E}_\mu(n)}(u) = \frac{\M_\nu\left(n+\tfrac{u}{\eta(n)}\right)}{\M_\nu(n)}\exp\left(-\frac{\M'_\nu(n)}{\M_\nu(n)}u\right).
\end{equation}
By applying Taylor's theorem with the Lagrange form of the remainder to the right-hand side we get
\begin{equation}
\label{eq:F-form-2}
\F_{\mathcal{E}_\mu(n)}(u)  = \exp\left(\frac{\eta^2\left(n+\frac{\Theta(u,n) u}{\eta(n)}\right)}{\eta^2(n)} \frac{u^2}{2}\right),
\end{equation}
where $\eta^2(n) = (\log \M_\nu(n))''$ and the mapping $\Theta$  is such that $|\Theta(u,n)| \leq 1$ for all $u \geq 0$ and $n \in \N$. Now, the fact that $\lim\limits_{n \to \infty} (\log \M_\nu(n)-G(n)) = 0$ allows us to use \cite[Theorem A]{balkema:1995} to conclude that $\frac{1}{\eta(n)} \stackrel{\infty}{\sim} s_G(n)$, where $s_G$ is the scale-function of $G$. By assumption $s_G$ is self-neglecting, and the self-neglecting property is closed under asymptotic equivalence, so we get that $1/\eta$ is self-neglecting. From \eqref{eq:F-form-2} it then follows that $\lim\limits_{n \to \infty}\F_{\mathcal{E}_\mu(n)}(u) = e^{u^2/2}$, where the convergence is uniform on bounded $u$-intervals. Since the convergence of the bi-lateral Laplace transform yields the convergence in distribution, we then also get that $\lim\limits_{n \to \infty}\F_{\mathcal{E}_\mu(n)}(iy) = e^{-y^2/2}$ uniformly on bounded $y$-intervals. However, note that substituting $iy$ for $u$ in \eqref{eq:F-form-1} gives
%As a bi-lateral Laplace transform it admits an analytic extension to the right-half plane, still denoted by $\F_{\mathcal{E}_\mu(n)}$, and thus, replacing $u$ by $iy$ in \eqref{eq:F-form-2} for any $y\in \R$, we get
%\begin{equation}
%\label{eq:normalized-Esscher-Fourier}
%\F_{\mathcal{E}_\mu(n)}(iy) = \exp\left(-\frac{\eta^2\left(n+i\frac{\theta y}{\eta(n)}\right)}{\eta^2(n)} \frac{y^2}{2}\right).
%\end{equation}
% From \eqref{eq:normalized-Esscher-Fourier} it then follows that $\lim_{n \to \infty}\F_{\mathcal{E}_\mu(n)}(iy) = e^{-y^2/2}$, where the convergence is uniform on bounded $y$-intervals. However, by replacing $u$ by $iy$ in \eqref{eq:F-form-1}, $\F_{\mathcal{E}_\mu(n)}$ may alternatively be characterized as
\begin{equation*}
\left|\F_{\mathcal{E}_\mu(n)}(iy)\right| = \left|\frac{\M_\nu\left(n+i\tfrac{y}{\eta(n)}\right)}{\M_\nu(n)}\right|.
\end{equation*}
Therefore, the assumption in \eqref{eq:analytic condition on Mellin transform} of \Cref{thm:(in)determinacy-moment-sequence} is that, for all $n \geq n_0$, we have $|\F_{\mathcal{E}_\mu(n)}(iy)| \leq h(y)$, for some $h \in \Leb^1(\R)$ and uniformly in $n$. By the dominated convergence theorem, we get the stronger convergence property $\lim\limits_{n \to \infty} ||\F_{\mathcal{E}_\mu(n)}(iy) - e^{-y^2/2}||_{\Leb^1(\R)} = 0$. This allows us to invoke \cite[Theorem 5.1]{balkema:1995}, from which we conclude that $\mu(dy) = \mu(y)dy$, $y\in \R,$ and that the continuous density $\mu$ has a Gaussian tail.  Then, Theorem 4.4 in the aforementioned paper allows us to identify the asymptotic behavior of $\mu$ as
\begin{equation*}
\mu(y) \stackrel{\infty}{\sim} \frac{1}{\sqrt{2\pi}}\frac{e^{-G_*(y)}}{s_{G_*}(y)} ,
\end{equation*}
where $G_*$ is the Legendre transform of $G$ and $s_{G_*}$ is its scale function. By changing variables it follows that $\nu(dx) = \nu(x)dx,x>0,$ and
\begin{equation*}
\nu(x) \stackrel{\infty}{\sim} \frac{1}{\sqrt{2\pi}}\frac{e^{-G_*(\log x)}}{xs_{G_*}(\log x)}
\end{equation*}
which completes the proof.
\end{proof}

\subsection{Proof of \Cref{thm:(in)determinacy-moment-sequence}} \label{subsec:proof-thm:(in)determinacy-moment-sequence}

First we use \Cref{prop:analogue-Tauberian-theorem} to get that $\nu(dx) = \nu(x)dx,x>0$, with
\begin{equation}
\label{eq:nu-asymptotic}
\nu(x) \stackrel{\infty}{\sim} \frac{1}{\sqrt{2\pi}}\frac{e^{-G_*(\log x)}}{xs_{G_*}(\log x)},
\end{equation}
where $G_* \in \apm$ is the Legendre transform of $G$ and $s_{G_*}$ is its scale function. To prove the indeterminacy of $\nu$ we aim to apply a refinement of Krein's theorem due to Pedersen \cite{pedersen:1998}, which amounts to showing that there exists $x_0 \geq 0$ such that
\begin{equation*}
\int_{x_0}^\infty \frac{-\log \nu(x^2)}{1+x^2} dx < \infty.
\end{equation*}
To this end, we recall that, by definition, $G_* \in \apm$ implies that $\lim\limits_{x \to \infty} G_*(x) = \infty$ and also that $\log B(x) \stackrel{\infty}{=} \littleo(G(x))$, which follows from \Cref{lem:self-neglecting-flat}\ref{item-2:lem:self-neglecting-flat}. These two facts mean that the asymptotic relation in \eqref{eq:nu-asymptotic} holds upon taking logarithms, so that one gets
\begin{equation*}
\log \nu(x) \stackrel{\infty}{\sim} -G_*(\log x) - \log s_{G_*}(\log x) - \log x - \log \sqrt{2\pi}.
\end{equation*}
Hence it suffices to show that
\begin{equation}
\label{eq:integral-shown}
\int_{x_0}^\infty \left(\frac{2\log x}{1+x^2} + \frac{\log s_{G_*}(\log x)}{1+x^2} + \frac{G_*(2\log x)}{1+x^2}\right) dx < \infty
\end{equation}
for $x_0$ chosen such that $G_*$ (and thus $s_{G_*}$) are well defined on $(x_0,\infty)$. Without loss of generality, we also suppose that $x_0$ is chosen large enough such that $G_*(\log x) > 0$ for all $x \geq x_0$, which is possible since $\lim\limits_{x \to \infty} G_*(\log x) = \infty$. Now, the integral of the first term in \eqref{eq:integral-shown} is plainly finite for any $x_0 \geq 0$. From \Cref{lem:self-neglecting-flat}\ref{item-1:lem:self-neglecting-flat} we have that $s_{G_*}(x) \stackrel{\infty}{=} \littleo(x)$ and therefore the integral of the second term in \eqref{eq:integral-shown} is also finite. Consequently it remains to bound the integral of the last term, for which, after performing a change of variables, we obtain
\begin{equation*}
\frac{1}{2}\int_{y_0}^\infty G_*(y)e^{-\frac{y}{2}} \frac{e^y}{1+e^y} dy \leq \int_{y_0}^\infty G_*(y) e^{-\frac{y}{2}} dy,
\end{equation*}
where $y_0 = 2\log x_0$. Then, using \eqref{eq:derivative-Legendre} and making
 another change of variables gives that
\begin{equation}
\label{eq:Krein-type integral}
\int_{y_0}^{\infty} G_*(y) e^{-\frac{y}{2}} dy = \int_{u_0}^{\infty} (uG'(u)-G(u)) G''(u) e^{-\frac{G'(u)}{2}} du,
\end{equation}
where $u_0 = \overleftarrow{G'}(y_0)$. The assumption in \Cref{thm:(in)determinacy-moment-sequence}\ref{item-1:thm:(in)determinacy-moment-sequence} is that the integral on the right is finite for some $x_0$ (and thus $u_0$) large enough, and thus we conclude that $\nu$ is moment indeterminate, which completes the proof of \Cref{item-1:thm:(in)determinacy-moment-sequence}.
For the proof of the next claim, we suppose that $G$ is defined on $(a,\infty)$ for some $a \geq - \infty$. Then, from the assumption in \Cref{item-a:thm:(in)determinacy-moment-sequence} and  the fact that $\frac{f^{-1/x}(x)}{g^{-1/x}(x)}=\left(\frac{f(x)}{g(x)}\right)^{-1/x}$ we get, for $A>a$ large enough, that
\begin{equation*}
\sum_{\max(1,\lceil A \rceil)}^\infty \M_\nu^{-\frac{1}{2n}}(n) \geq C_1 \sum_{\max(1,\lceil A \rceil)}^\infty e^{-\frac{G(n)}{2n}}
\end{equation*}
for some constant $C_1 > 0$. Since $G$ is convex and differentiable it satisfies, for all $n,r \in (A,\infty)$, the inequality
\begin{equation*}
G(n) - G(r) \leq G'(n)(n-r).
\end{equation*}
Let $r>0$ be fixed such that $G(r) > 0$, noting that it is possible to find such an $r$ by the fact that $G = (G_*)_* \in \apm$ and hence $\lim\limits_{u \to \infty} G(u) = \infty$, see \Cref{lem:admissible G closed under Legendre transform}. Using this fact and that $G\in \mathrm{C}_+^2((a,\infty))$, we get, from the above inequality, that
\begin{equation*}
\sum_{\max(1,\lceil A \rceil)}^\infty e^{-\frac{G(n)}{2n}} \geq \sum_{\max(1,\lceil A \rceil)}^\infty e^{\frac{G'(n)r - G(r)}{2n}}e^{-\frac{G'(n)}{2}} \geq C_2 \sum_{\max(1,\lceil A \rceil)}^\infty e^{-\frac{G(r)}{2n}} e^{-\frac{G'(n)}{2}} \geq C_3 \sum_{\max(1,\lceil A \rceil)}^\infty e^{-\frac{G'(n)}{2}},
\end{equation*}
where $C_2,C_3 >0 $ are constants. Putting all of these observations together gives that
\begin{equation}
\label{eq:sum-G'}
\sum_{\max(1,\lceil A \rceil)}^\infty \M_\nu^{-\frac{1}{2n}}(n) \geq C \sum_{\max(1,\lceil A \rceil)}^\infty e^{-\frac{G'(n)}{2}}
\end{equation}
for some constant $C > 0$. We wish to compare the sum in the right-hand side of \eqref{eq:sum-G'} with the integral in \eqref{eq:Krein-type integral}. To this end we integrate the right-hand side of \eqref{eq:Krein-type integral} by parts to obtain
\begin{equation}
\label{eq:IBP-1}
\int_{u_0}^{\infty} (uG'(u)-G(u)) G''(u) e^{-\frac{G'(u)}{2}} du = -2(uG'(u)-G(u))e^{-\frac{G'(u)}{2}} \Big|_{u_0}^{\infty} + 2 \int_{u_0}^\infty uG''(u)e^{-\frac{G'(u)}{2}} du.
\end{equation}
From the assumption that $\lim\limits_{u \to \infty} ue^{-\frac{G'(u)}{2}} < \infty$ an application of L'H\^{o}pital's rule allows us to conclude that
\begin{equation*}
\label{eq:limit-calculation}
\lim\limits_{u \to \infty} \frac{uG'(u)-G(u)}{e^\frac{G'(u)}{2}} = 2\lim\limits_{u \to \infty} \frac{u}{e^\frac{G'(u)}{2}} < \infty
\end{equation*}
and hence the boundary term in \eqref{eq:IBP-1} is a finite constant, say $b_1$. Integrating by parts again the right-hand integral in \eqref{eq:IBP-1} we obtain
\begin{equation}
\label{eq:IBP-2}
b_1 + 2 \int_{u_0}^\infty uG''(u)e^{-\frac{G'(u)}{2}} du = b_1 - 4ue^{-\frac{G'(u)}{2}}\Big|_{u_0}^\infty + 4\int_{u_0}^\infty e^{-\frac{G'(u)}{2}} du.
\end{equation}
The limit at infinity for the boundary term in \eqref{eq:IBP-2} above can be controlled by the assumption that $\lim\limits_{u \to \infty} ue^{-\frac{G'(u)}{2}} < \infty$ and hence this boundary term, say $b_2$, is also finite. Thus, we get
\begin{equation*}
\int_{u_0}^{\infty} (uG'(u)-G(u)) G''(u) e^{-\frac{G'(u)}{2}} du = b_1 + b_2 + 4\int_{u_0}^\infty e^{-\frac{G'(u)}{2}} du.
\end{equation*}
Now, since $G'$ is non-decreasing we have, taking $u_0 \geq \max(1,\lceil A \rceil)$,
\begin{equation*}
\int_{u_0}^\infty e^{-\frac{G'(u)}{2}} du \leq \sum_{n=\lfloor u_0 \rfloor}^\infty e^{-\frac{G'(n)}{2}},
\end{equation*}
so that finally we establish the inequalities
\begin{equation}
\label{eq:three-inequalities}
\int^\infty(uG'(u)-G(u)) G''(u) e^{-\frac{G'(u)}{2}} du = b_1 + b_2 + 4\int^\infty e^{-\frac{G'(u)}{2}} du \leq b_1+b_2 + \frac{4}{C} \sum^\infty \M_\nu^{-\frac{1}{2n}}(n),
\end{equation}
where $b_1,b_2 \in \R$ and $C > 0$ are finite constants. Hence, if the Carleman's sum on the right-hand side of \eqref{eq:three-inequalities} is finite we conclude that $\int^\infty e^{-\frac{G'(u)}{2}}du < \infty$, which implies that $\int^\infty (uG'(u)-G(u)) G''(u) e^{-\frac{G'(u)}{2}} du < \infty$, which in turn yields the moment indeterminacy of $\nu$. Conversely, if the integral on the left-hand side of \eqref{eq:three-inequalities} is infinite then $\int^\infty e^{-\frac{G'(u)}{2}}du = \infty$, which forces Carleman's sum to be infinite, thus yielding the moment determinacy of $\nu$. This completes the proof of \Cref{item-2:thm:(in)determinacy-moment-sequence} and hence of the theorem.

%\subsection{Proof of \Cref{cor:indeterminacy-moment-sequence-factorization}} \label{subsec:proof-cor:indeterminacy-moment-sequence-factorization}
%
%Since $\nu$ satisfies the assumptions of \Cref{thm:(in)determinacy-moment-sequence} it follows by \Cref{thm:(in)determinacy-moment-sequence}\ref{item-1:thm:(in)determinacy-moment-sequence} that
%\begin{equation}
%\int^\infty (xG'(x)-G(x)) G''(x) e^{-\frac{G'(x)}{2}}  dx < \infty \implies \nu \text{ is moment indeterminate.}
%\end{equation}
%However, since we have, for $n \geq 0$,
%\begin{equation}
%\M_\mathcal{V}(n) = \M_\nu(n) \hspace{1pt} \mathfrak{m}(n),
%\end{equation}
%where $(\mathfrak{m}(n))_{n \geq 0}$ is a non-vanishing moment sequence, we may apply \cite[Lemma 2.2 and Remark 2.3]{berg:2004} to get that the indeterminacy of $\nu$ implies the indeterminacy of $\mathcal{V}$.
%
%\qed

\subsection{Proof of \Cref{thm:(in)determinacy-density}} \label{subsec:proof-thm:(in)determinacy-density}

% Need to change everything in here to G_*

We first note that it suffices to prove \Cref{thm:(in)determinacy-density}\ref{item-1:thm:(in)determinacy-density} under the assumption that $\nu(dx) = \nu(x)dx$ such that
\begin{equation*}
\nu(x) \stackrel{\infty}{=} \bigO(e^{-G_*(\log x)}).
\end{equation*}
This is because establishing the result in this case allows us to apply it to the probability density
\begin{equation*}
x \mapsto \frac{\overline{\nu}(x)}{\M_\nu(1)},
\end{equation*}
which is the density of the so-called stationary-excess distribution of $\nu$ (of order 1), and whose moment determinacy implies the moment determinacy of $\nu$, see~\cite[Section 2]{berg:2004a}. Hence, we suppose that there exist constants $A', C' > 0$ such that, for  all  $x \geq A'$,
\begin{equation*}
\nu(x) \leq C' e^{-G_*(\log x)}.
\end{equation*}
 Without loss of generality we can replace $G_*(x)$ by $G_*(x) - x$, since the addition of linear functions does not affect the asymptotically parabolic property or the other conditions of the theorem. Hence we may assume that there exist constants $A, C > 0$ such that,  for  all  $x \geq A$,
\begin{equation} \label{eq:bound_nu}
\nu(x) \leq C e^{-G_*(\log x)-\log x}.
\end{equation}
 Set, for $n \geq 0$,
\begin{equation*}
\mathfrak{s}(n) = \int_A^\infty x^n e^{-G_*(\log x)-\log x} dx.
\end{equation*}
By a change of variables
\begin{equation}
\label{eq:change-of-vars-s}
\mathfrak{s}(n) = \int_{\log A}^\infty e^{ny-G_*(y)}dy
\end{equation}
and since $G_* \in \apm$ the right-hand side is finite for all $n \geq 0$. This combined with \eqref{eq:bound_nu} allows us to obtain the bound, for $n \geq 0$,
\begin{equation*}
\M_\nu(n) = \int_0^\infty x^n \nu(x) dx \leq \int_0^A x^n \nu(x)dx + C\int_A^{\infty} x^n e^{-G_*(\log x) - \log x} dx \leq A^n + C\mathfrak{s}(n).
\end{equation*}
 Since $\supp(\nu) = [0,\infty)$ it is straightforward that $A^n \stackrel{\infty}{=} \littleo(\M_\nu(n))$ and hence, for $n$ large enough,
\begin{equation*}
C\mathfrak{s}(n) \geq \M_\nu(n) - A^n = \M_\nu(n)\left(1-\frac{A^n}{\M_\nu(n)}\right)  \geq c \M_\nu(n),
\end{equation*}
where $c \in (0,1)$ is a constant. Therefore to prove moment determinacy it suffices to show the divergence of the sum $\sum^\infty \mathfrak{s}^{-\frac{1}{2n}}(n)$ for a suitable lower index, since this would imply the divergence of Carleman's sum $\sum^\infty \M_\nu^{-\frac{1}{2n}}(n)$. To this end, we note that $G_* \in \apm$ implies that the function
\begin{equation*}
f(y) = e^{-G_*(y)}\mathbb{I}_{\{y > \log A\}}
\end{equation*}
satisfies the conditions of \cite[Theorem B]{balkema:1995}, see e.g.~the proof of \Cref{prop:analogue-Tauberian-theorem}. The expression for $\mathfrak{s}$ given in \eqref{eq:change-of-vars-s} allows us to invoke this quoted result to conclude that
\begin{equation*}
\mathfrak{s}(n) \stackrel{\infty}{\sim} \frac{\sqrt{2\pi}}{s_{G}(n)}e^{G(n)},
\end{equation*}
where $(G_*)_* = G$ is the Legendre transform of $G_*$ and $s_{G}$ is its scale function. Hence, choosing $a$ large enough with $G$ well-defined on $(a,\infty)$ and say $\sqrt{2\pi}\leq e^{2\lceil a \rceil}$, we have
\begin{equation*}
\sum_{\max(\lceil a \rceil, 1)}^\infty \mathfrak{s}^{-\frac{1}{2n}}(n) \geq C_1 \sum_{\max(\lceil a \rceil, 1)}^\infty s_{G}(n)^{\frac{1}{2n}} e^{-\frac{G(n)}{2n}}
\end{equation*}
for some constant $C_1 > 0$. Now, from \eqref{eq:derivative-Legendre} we have the relation $G(G_*'(n)) = nG_*'(n) - G_*(n)$ and differentiating this relation twice gives that $s_G(G_*'(n)) = 1 / s_{G_*}(n)$. Using these facts and recalling that $\overleftarrow{G_*'}$ denotes the inverse of $G_*'$, we deduce that
\begin{equation*}
\sum_{\max(\lceil a \rceil, 1)}^\infty s_{G}(n)^{\frac{1}{2n}} e^{-\frac{G(n)}{2n}} = \sum_{\max(\lceil a \rceil, 1)}^\infty s_{G_*}(\overleftarrow{G_*'}(n))^{-\frac{1}{2n}} e^{\frac{G_*(\overleftarrow{G_*'}(n))}{2n}} e^{-\frac{\overleftarrow{G_*'}(n)}{2}}.
\end{equation*}
Next, we recall from \Cref{lem:self-neglecting-flat}\ref{item-2:lem:self-neglecting-flat} that $ \log s_{G_*}(n) \stackrel{\infty}{=} \littleo(G_*(n))$ and note that $\overleftarrow{G_*'}$ is an increasing function satisfying $\lim\limits_{n \to \infty} \overleftarrow{G_*'}(n) = \infty$. Thus, given a constant $C_2 \in (0,1)$ it follows that there exists $N \in \N$ such that, for $n \geq N$,
\begin{equation*}
G_*(\overleftarrow{G_*'}(n))\left(1-\frac{\log s_{G_*}(\overleftarrow{G_*'}(n))}{G_*(\overleftarrow{G_*'}(n))}\right) \geq C_2 G_*(\overleftarrow{G_*'}(n)).
\end{equation*}
As $G_* \in \apm$, we get that, for some constant $C_3 > 0$,
\begin{equation*}
\sum_{\max(\lceil a \rceil, 1)}^\infty s_{G_*}(\overleftarrow{G_*'}(n))^{-\frac{1}{2n}} e^{\frac{G_*(\overleftarrow{G_*'}(n))}{2n}} e^{-\frac{\overleftarrow{G_*'}(n)}{2}} \geq C_3 \sum_{\max(\lceil a \rceil, 1)}^\infty e^{-\frac{\overleftarrow{G_*'}(n)}{2}}.
\end{equation*}
Putting all of these observations together gives us the inequality
\begin{equation*}
\sum_{\max(\lceil a \rceil, 1)}^\infty \mathfrak{s}(n)^{-\frac{1}{2n}} \geq \tilde{C} \sum_{\max(\lceil a \rceil, 1)}^\infty e^{-\frac{\overleftarrow{G_*'}(n)}{2}}
\end{equation*}
for some constant $\tilde{C} > 0$ depending only on $A$. If the sum on the right-hand side diverges, which is the condition of \Cref{item-1:thm:(in)determinacy-density}, it follows that $\nu$ is moment determinate, which concludes the proof in this case.
Next, for the proof of \Cref{item-2:thm:(in)determinacy-density}, we again assume, without loss of generality, that,  for $x \geq A$,
\begin{equation*}
\frac{1}{C} e^{-G_*(\log x)-\log x} \leq \nu(x) \leq C e^{-G_*(\log x)-\log x}
\end{equation*}
for some constants $C,A > 0$ (which may be different from the ones above). Then, for another constant $K > 0$,
\begin{align*}
\int_A^\infty \frac{-\log \nu(x^2)}{1+x^2}dx &\leq \int_A^\infty \frac{G_*(2\log x)+\log(Cx^2)}{1+x^2}dx \\ &\leq  \int_A^\infty \frac{G_*(2\log x)}{1+x^2}dx +K =  \int_{2\log A}^\infty G_*(y)\frac{e^{\frac{y}{2}}}{2(1+e^y)} dy +K\\
&\leq  \frac{1}{2(1+A^{-2})} \int_{2\log A}^\infty G_*(y)e^{-\frac{y}{2}}dy +K.
\end{align*}
Thus it suffices to show that if the right-most integral is infinite then $\nu$ is moment determinate. To this end, we wish to perform a similar integration by parts calculation as in the proof of \Cref{thm:(in)determinacy-moment-sequence}, which requires us to have
\begin{equation*}
\lim\limits_{u \to \infty} ue^{-\frac{{G}'(u)}{2}} < \infty.
\end{equation*}
As noted earlier, \eqref{eq:derivative-Legendre} yields $G(G_*'(x)) = xG_*'(x) - G_*(x)$ and by differentiating this relation one gets that $G_*'$ and $G'$ are inverses of each other. Hence,
\begin{equation*}
\lim\limits_{u \to \infty} ue^{-\frac{{G}'(u)}{2}} < \infty \iff \lim_{x \to \infty} G_*'(x)e^{-\frac{x}{2}} < \infty
\end{equation*}
and the rest of the proof proceeds as in the proof of \Cref{thm:(in)determinacy-moment-sequence}.

\subsection{Proof of \Cref{cor:generalized-Hardy}} \label{subsec:proof-cor:generalized-Hardy}

By assumption we have that, for some $c > 0$, $\int_0^\infty e^{c\sqrt{x}} \nu(dx) < \infty$. Write $\rho(x) = \int_0^x e^{c\sqrt{y}} \nu(dy)$, $x \geq 0$, so that the assumption is equivalent to $\rho(\infty) = \lim\limits_{x \to \infty} \rho(x) < \infty$. To obtain the desired estimate we adapt the arguments of the proof of \cite[Theorem 2.2b, Chapter II]{widder:1941} to our setting. From \cite[Theorem 6b, Chapter I]{widder:1941} we get that
\begin{equation*}
\overline{\nu}(x) = \int_x^\infty e^{-c\sqrt{y}} \rho(dy)
\end{equation*}
and integrating the right-hand side of the above by parts gives
\begin{equation*}
\overline{\nu}(x) = \lim_{y \to \infty}  e^{-c\sqrt{y}}\rho(y) -  e^{-c\sqrt{x}}\rho(x) + \int_x^\infty \frac{c}{2\sqrt{y}} e^{-c\sqrt{y}}\rho(y)dy = -e^{-c\sqrt{x}}\rho(x) + \int_x^\infty \frac{c}{2\sqrt{y}} e^{-c\sqrt{y}}\rho(y)dy,
\end{equation*}
where the second equality follows readily from $\rho(\infty) < \infty$. Consequently
\begin{align*}
\lim_{x \to \infty} e^{c\sqrt{x}} \overline{\nu}(x) &= -\rho(\infty) + \lim_{x \to \infty} e^{c\sqrt{x}} \int_x^\infty \frac{c}{2\sqrt{y}} e^{-c\sqrt{y}}\rho(y)dy \\
&= -\lim_{x \to \infty} e^{c\sqrt{x}} \int_x^\infty \frac{c}{2\sqrt{y}} e^{-c\sqrt{y}}(\rho(\infty)- \rho(y))dy \\
&= - \lim_{x \to \infty} \int_x^\infty \frac{c}{2\sqrt{y}} e^{-c(\sqrt{y}-\sqrt{x})}(\rho(\infty)- \rho(y))dy.
\end{align*}
To evaluate the limit on the right we make the change of variables $y \mapsto (y + \sqrt{x})^2$, which gives that
\begin{equation*}
\int_x^\infty \frac{c}{2\sqrt{y}} e^{-c(\sqrt{y}-\sqrt{x})}(\rho(\infty)- \rho(y))dy = \int_0^\infty ce^{-cy}( \rho(\infty)-\rho((y+\sqrt{x})^2)) dy.
\end{equation*}
As $\rho$ is positive and increases monotonically to $\rho(\infty)$ it follows that the integrand is uniformly (in $x$) dominated by the integrable function $y \mapsto c(\rho(\infty)-\nu(\{0\}))e^{-cy}$ so by the dominated convergence theorem we get that
\begin{equation*}
\lim_{x \to \infty} e^{c\sqrt{x}} \overline{\nu}(x)  = -\lim_{x \to \infty} \int_0^\infty ce^{-cy}( \rho(\infty)-\rho((y+\sqrt{x})^2)) dy = 0
\end{equation*}
which, in particular, implies that $\overline{\nu}(x) \stackrel{\infty}{=} \bigO(e^{-c\sqrt{x}})$. Conversely suppose that, for $x \geq A$, $\overline{\nu}(x) \leq K e^{-c\sqrt{x}}$ for some $K, c, A > 0$. Then, for $0 < c' < c$,
\begin{equation*}
\int_0^\infty e^{c'\sqrt{x}} \nu(dx) \leq e^{c'\sqrt{A}} + \int_A^\infty e^{c'\sqrt{x}}\nu(dx).
\end{equation*}
An integration by parts yields % applying Fubini's theorem to the integral on the right-hand side we get \textcolor{red}{I do not understand the identity below}
\begin{align*}
\int_A^\infty e^{c'\sqrt{x}}\nu(dx) &= e^{c'\sqrt{A}}\overline{\nu}(A) - \lim_{x\to \infty}e^{c'\sqrt{x}} \overline{\nu}(x) + \int_A^{\infty} \frac{c'}{2\sqrt{x}} e^{c'\sqrt{x}} \overline{\nu}(x) dx \\
&\leq  e^{c'\sqrt{A}}\overline{\nu}(A)+  K\int_A^\infty \frac{c'}{2\sqrt{x}}e^{(c'-c)\sqrt{x}} dx < \infty
\end{align*}
and thus Hardy's condition is satisfied for any $c' \in (0,c)$. The fact that $x \mapsto ce^{\frac{x}{2}} \in \apm$ is readily checked, which completes the proof of the first item.
Next, from \eqref{eq:big O Hardy example}, there exist $M,\underline{x} > 0$ such that, for $x \geq \underline{x}$,
\begin{equation*}
\nu(x) \geq Me^{-\frac{\alpha\sqrt{x}}{\log x}}.
\end{equation*}
Consequently, for $c > 0$,
\begin{equation*}
\int_0^\infty e^{c\sqrt{x}} \nu(x) dx \geq M\int_{\underline{x}}^\infty e^{\sqrt{x}\left(c-\frac{\alpha}{\log x}\right)} dx
\end{equation*}
and clearly the right-hand side is infinite for any $c > 0$, meaning that the Hardy's condition is not satisfied by $\nu$. Next, write $G_*(x) = \alpha x^{-1}e^{\frac{x}{2}}$ and note that, by \eqref{eq:big O Hardy example}, we have
\begin{equation*}
\nu(x) \stackrel{\infty}{\asymp} e^{-G_*(\log x)}.
\end{equation*}
To show moment determinacy of $\nu$, we will show that $G_*$ satisfies the assumptions of \Cref{thm:(in)determinacy-density}\ref{item-2:thm:(in)determinacy-density}. A straightforward computation gives
\begin{equation*}
s_{G_*}(x) = \frac{2e^{-\frac{x}{4}}}{\sqrt{\alpha f(x)}}, \text{ where } f(x) = \left(\frac{1}{x} - \frac{1}{x^2} + \frac{8}{x^3}\right),
\end{equation*}
which is plainly positive for $x > 0$. Since $\ap$ is a convex cone, $f(x) \stackrel{\infty}{\sim} x^{-1}$, and self-neglecting functions are closed under asymptotic equivalence, it suffices to show that the function
\begin{equation*}
s(x) = \sqrt{x}e^{-\frac{x}{4}}
\end{equation*}
is self-neglecting. However, this is immediate, as $\lim\limits_{x \to \infty} s(x) = 0$ and
\begin{equation*}
\frac{s(x+ws(x))}{s(x)} = e^{-ws(x)}\sqrt{1+\frac{ws(x)}{x}}.
\end{equation*}
Next, since $G_*'(x) = \alpha x^{-1}e^{\frac{x}{2}}\left(\frac{1}{2}-\frac{1}{x}\right)$ it follows that $\lim\limits_{x\to\infty} G_*'(x)e^{-\frac{x}{2}} = 0$. Finally, for any $x_0 \geq 0$,
\begin{equation*}
\int_{x_0}^\infty G_*(x)e^{-\frac{x}{2}}dx = \alpha \int_{x_0}^\infty \frac{1}{x}dx = \infty,
\end{equation*}
so that by \Cref{thm:(in)determinacy-density}\ref{item-2:thm:(in)determinacy-density} $\nu$ is moment determinate for all $\alpha > 0$.

\subsection{Proof of \Cref{thm:threshold-log-Levy}} \label{subsec:proof-thm:threshold-log-Levy}

\subsubsection{Proof of \Cref{thm:threshold-log-Levy}\ref{item-1:thm:threshold-log-Levy}}

First we shall prove the claim in the case $\overline{\Pi}(0)=\int_0^{\infty}\Pi(dr) = 0$ and $\sigma^2>0$. Since $Y = (Y_t)_{t \geq 0}$ admits all exponential moments its characteristic exponent $\Psi$ admits an analytical extension to the right-half plane, which we still denote by $\Psi$, and takes the form \eqref{eq:LK-representation} for $u \geq 0$. %Hence, for any $u\geq 0$,
%\begin{equation*}
%\E[e^{uY_t}] = e^{t\Psi(u)},
%\end{equation*}
%where we recall that $
%\Psi(u) = bu + \frac{\sigma^2}{2}u^2 + \int_{-\infty}^0 (e^{ur}-1-ur\mathbb{I}_{\{|r| \leq 1\}})\Pi(dr)$
%with $\sigma > 0$.
Let $t > 0$ be fixed. Differentiating $\Psi$ in \eqref{eq:LK-representation}, see e.g.~\cite[p.~347]{sato:2013}, one gets
\begin{equation} \label{eq:psipp}
\Psi'(u) = b + \sigma^2 u + \int_{-\infty}^0 \left(e^{ur} - \mathbb{I}_{\{|r|\leq 1\}}\right) r\Pi(dr) \quad \textrm{and} \quad \Psi''(u) = \sigma^2 + \int_{-\infty}^0 r^2e^{ur} \Pi(dr) >0,
\end{equation}
where the integrability conditions on $\Pi$  also ensure that $\Psi''$ is well-defined on $\R_+$. Next, invoking  the dominated convergence theorem, we have $\lim\limits_{u \to \infty} \int_{-\infty}^0 r^2e^{ur} \Pi(dr) = 0$ and hence $\frac{1}{\sqrt{\Psi''(u)}} \stackrel{\infty}{\sim} \frac{1}{\sigma}$. Since constants are trivially self-neglecting and self-neglecting functions are closed under asymptotic equivalence, it follows that $\Psi \in \ap$. Furthermore, since $\ap$ is a convex cone we get that $t\Psi \in \ap$, and thus the condition in \Cref{thm:(in)determinacy-moment-sequence}\ref{item-a:thm:(in)determinacy-moment-sequence} is fulfilled.
Let us now write $\nu_t(dx) = \mathbb P( e^{Y_t} \in dx),x>0$, and, for all $n\geq 1 $,
\begin{equation} \label{eq:def_eta}
\eta^2(n) = (\log \M_{\nu_t}(n))''= t\Psi''(n).
\end{equation} Then, for all $n\geq 1 $ and $y \in \R$, we have
\begin{eqnarray*}
\log \left|\frac{\M_{\nu_t}\left(n+i\tfrac{y}{\eta(n)}\right)}{\M_{\nu_t}(n)}\right| &=& t\Re\left(\Psi\left(n+i\tfrac{y}{\eta(n)}\right) - \Psi(n)\right) \\ &=&  -\frac{t\sigma^2}{2\eta^2(n)}y^2 + t\int_{-\infty}^0 e^{nr}\left(\cos\left(\tfrac{yr}{\eta(n)}\right)-1\right) \Pi(dr) \\ & \leq & -\frac{t\sigma^2}{2\eta^2(n)}y^2,
\end{eqnarray*}
where we simply use the trivial bound for the integral term. By combining \eqref{eq:def_eta} with \eqref{eq:psipp}, one easily gets that, for any $n \geq 1$, $\eta^2(n)\geq t\sigma^2$ and thus
%\begin{equation}
%t\Re\left[\Psi\left(n+\tfrac{iy}{\eta(n)}\right) - \Psi(n)\right] \leq -\frac{t\sigma^2}{2\eta^2(n)}y^2 \leq -\frac{y^2}{2}.
%\end{equation}
% Hence, for $n \geq 1$,
\begin{equation*}
\left|\frac{\M_{\nu_t}\left(n+\tfrac{iy}{\eta(n)}\right)}{\M_{\nu_t}(n)}\right|  \leq e^{-\frac{y^2}{2}}
\end{equation*}
which shows that the condition of \Cref{thm:(in)determinacy-moment-sequence}\ref{item-b:thm:(in)determinacy-moment-sequence} is satisfied. From \eqref{eq:psipp}, one observes, since $\sigma^2>0$, that
\begin{equation}\label{eq:asympt_psip}
\Psi''(u) \stackrel{\infty}{\sim} \sigma^2\end{equation} and thus by integration, see \cite[Section 1.4]{Olver-74}, $\Psi'(u) \stackrel{\infty}{\sim} \sigma^2 u$. Therefore
\begin{equation*}
\lim\limits_{u \to \infty} u e^{-\frac{t\Psi'(u)}{2}} < \infty.
\end{equation*}
Finally, note that $\Psi(n)= \frac{\sigma^2}{2}n^2+\Psi_0(n)$ where $\Psi_0$ is the Laplace exponent of another spectrally negative L\'evy process (possibly the negative of a subordinator). By convexity of $\Psi_0$ and the fact that $\lim\limits_{n\to \infty} \Psi_0(n)=\infty$, we have $\Psi_0(n)>0$ for $n$ large enough and thus we obtain the following upper bound
\begin{equation*}
\sum^\infty \M_{\nu_t}^{-\frac{1}{2n}}(n) \leq \sum^\infty e^{-\frac{t\sigma^2}{4} n} < \infty.
\end{equation*}
By \Cref{thm:(in)determinacy-moment-sequence}\ref{item-2:thm:(in)determinacy-moment-sequence} it follows that $\nu_t$, the law of $ e^{Y_t}$, is moment indeterminate for all $t > 0$.
In the general case when $\overline{\Pi}(0) > 0$ we may separate the terms and write
\begin{align*}
\Psi(u) &= bu + \frac{1}{2}\sigma^2u^2 + \int_{-\infty}^0 (e^{ur}-1-ur\mathbb{I}_{\{|r| \leq 1\}})\Pi(dr) + \int_0^\infty (e^{ur}-1-ur\mathbb{I}_{\{|r| \leq 1\}})\Pi(dr) \\
&= \Psi_-(u) + \int_0^\infty (e^{ur}-1-ur\mathbb{I}_{\{|r| \leq 1\}})\Pi(dr) \\
&= \Psi_-(u) + \Psi_+(u),
\end{align*}
where $\Psi_-$ and $\Psi_+$  are the characteristic exponents of a L\'evy process and the  L\'evy measure associated to $\Psi_-$, say $\Pi_-$, satisfies $\overline{\Pi}_-(0) = 0$. Thus, for any $n \geq 0$,
\begin{equation*}
\M_{\nu_t}(n) = \int_0^\infty x^n \ \mathbb{P}(X_t \in dx) = \int_{-\infty}^\infty e^{ny} \ \mathbb{P}(Y_t \in dy) = e^{t\left(\Psi_-(n) + \Psi_+(n)\right)}
\end{equation*}
and from the earlier observations $(e^{t\Psi_-(n)})_{n \geq 0}$ is an indeterminate moment sequence. Since $e^{t\Psi_+(n)} > 0$ for all $n, t \geq 0$, Corollary \ref{cor:indeterminacy-moment-sequence-factorization} gives that the random variable $e^{Y_t}$ is moment indeterminate for all $t > 0$.

\subsubsection{Proof of \Cref{thm:threshold-log-Levy}\ref{item-2:thm:threshold-log-Levy}}
First, for any $n, t \geq 0$, writing again $\nu_t(dx) = \mathbb P(e^{Y_t}\in dx),x>0$, we have
\begin{equation*}
\M_{\nu_t}(n) = e^{tn\log(n+1)}
\end{equation*}
and hence
\begin{equation*}
\sum^\infty \M_{\nu_t}^{-\frac{1}{2n}}(n) = \sum^\infty (n+1)^{-\frac{t}{2}}.
\end{equation*}
The latter series diverges if and only if $t \leq 2$, which by Carleman's criterion yields the moment determinacy of $\nu_t$ for $t \leq 2$.
For the proof of indeterminacy we resort to an application of \Cref{thm:(in)determinacy-moment-sequence}. To this end, we first check that the function $\Psi(u) = u\log(u+1)$ is asymptotically parabolic on $\R^+$. Plainly, $\Psi$ is twice differentiable and  taking derivatives we have, for any $u\geq 0$,
\begin{equation}
\Psi''(u) = \frac{u+2}{(u+1)^2}> 0
\end{equation}
and thus $\Psi \in \mathrm{C}_+^2(\R_+)$. Clearly $\frac{1}{\sqrt{\Psi''(u)}} \stackrel{\infty}{\sim} \sqrt{u}$ and it is readily checked that $u \mapsto \sqrt{u}$ is self-neglecting. Since self-neglecting functions are closed under asymptotic equivalence it follows that $\Psi \in \ap$ and since $\ap$ is a convex cone we get that $ t\Psi \in \ap$, for any $t > 0$.
We proceed by verifying  that the condition in \Cref{thm:(in)determinacy-moment-sequence}\ref{item-b:thm:(in)determinacy-moment-sequence} is fulfilled for all $t > 0$. Write $\Log:\C \to \C$ for the holomorphic branch of  the complex logarithm such that $\Log(1) = 0$ and let $\eta$ be defined by
\begin{equation*}
\eta^2(n) = (\log \M_{\nu_t}(n))'' = t\frac{n+2}{(n+1)^2}.
\end{equation*}
Then, for all $n \in \N$ and $y\in \R$,
\begin{equation*}
\left|\frac{\M_{\nu_t}\left(n+\tfrac{iy}{\eta(n)}\right)}{\M_{\nu_t}(n)}\right| = e^{t\Re\left( \left(n+\tfrac{iy}{\eta(n)}\right)\Log\left(n+1+\tfrac{iy}{\eta(n)}\right) - n\log(n+1)\right)}.
\end{equation*}
Focusing on the term inside the exponential, we have
\begin{align*}
&\Re\left( \left(n+\tfrac{iy}{\eta(n)}\right)\Log\left(n+1+\tfrac{iy}{\eta(n)}\right)  - n\log(n+1) \right) = \\  &n\log\left(\frac{\sqrt{(n+1)^2 + \frac{y^2}{\eta^2(n)}}}{(n+1)}\right) - \frac{y}{\eta(n)}\arctan\left(\frac{y}{{\eta(n)}(n+1)}\right).
\end{align*}
Simplifying within the logarithm and substituting for the definition of $\eta$ then yields
\begin{align*}
&\Re\left( \left(n+\tfrac{iy}{\eta(n)}\right)\Log\left(n+1+\tfrac{iy}{\eta(n)}\right) - n\log(n+1)\right) = \\ &\frac{n}{2}\log\left(1+\frac{y^2}{t(n+2)}\right) - \frac{y(n+1)}{\sqrt{t(n+2)}}\arctan\left(\frac{y}{\sqrt{t(n+2)}}\right).
\end{align*}
Since $\log(1+x^{-1}) \stackrel{\infty}{=} x^{-1}+\littleo(x^{-1})$ (resp.~$\arctan(x^{-1}) \stackrel{\infty}{=} x^{-1}+\littleo(x^{-1})$) , we have that
\begin{equation*}
\lim_{n \to \infty} \frac{n}{2}\log\left(1+\frac{y^2}{t(n+2)}\right) = \frac{y^2}{2t} \quad \left( \textrm{resp. }  \lim_{n \to \infty} \frac{y(n+1)}{\sqrt{t(n+2)}}\arctan\left(\frac{y}{\sqrt{t(n+2)}}\right) = \frac{y^2}{t}\right).
\end{equation*}
It follows that there exists $n_0>0$ such that for all $n\geq n_0$,
\begin{equation*}
\left|\frac{\M_{\nu_t}\left(n+\tfrac{iy}{\eta(n)}\right)}{\M_{\nu_t}(n)}\right|  \leq e^{-Cy^2},
\end{equation*}
where $0< C <\frac{1}{2}$ is a constant depending only on $n_0$. Hence the integrability condition in \Cref{thm:(in)determinacy-moment-sequence}\ref{item-b:thm:(in)determinacy-moment-sequence} is satisfied for any $t > 0$.
The proof will be completed if we can show that the additional condition in \Cref{thm:(in)determinacy-moment-sequence}\ref{item-2:thm:(in)determinacy-moment-sequence} holds, namely that
\begin{equation*}
\lim\limits_{u \to \infty} ue^{-\frac{t\Psi'(u)}{2}} <\infty \quad \textrm{for} \quad t\geq2.
\end{equation*}
However, simple algebra yields that for $t \geq 2$
\begin{equation*}
\lim\limits_{u \to \infty} ue^{-\frac{t\Psi'(u)}{2}} = \lim\limits_{u \to \infty} ue^{-\frac{t}{2}\left(\frac{u}{u+1} + \log(u+1)\right)} =\lim\limits_{u \to \infty} u(u+1)^{-\frac{t}{2}}e^{-\frac{tu}{2(u+1)}} < \infty.
\end{equation*}

\subsection{Proof of \Cref{cor:asymptotic-spectrally-negative}} \label{subsec:proof-cor:asymptotic-spectrally-negative}
In this case, we write $\nu_t(x)dx = \mathbb P(e^{Y_t}\in dx),x>0$, see the comments before the statement. In the proof of \Cref{thm:threshold-log-Levy}\ref{item-1:thm:threshold-log-Levy} it was shown that, for any $t > 0$, $\M_{\nu_t}$ fulfills the assumptions of \Cref{prop:analogue-Tauberian-theorem} when $\sigma^2 > 0$ and $\overline{\Pi}(0) = 0$. Invoking this result, noting that $(t\Psi)_*(y) = t\Psi_*(\frac{y}{t})$, and changing variables, we get for $f_t(y) dy= \mathbb{P} (Y_t \in dy),y \in \R$, and any $t > 0$,
\begin{equation*}
f_t(y) \stackrel{\infty}{\sim} \frac{1}{\sqrt{2\pi t}} \sqrt{\Psi_{*}''\left(\frac{y}{t}\right)} e^{-t\Psi_*\left(\frac{y}{t}\right)}.
\end{equation*}
Next, we have, for $y > 0$,
\begin{eqnarray*}
\Psi_*(\Psi'(y)) &=& y\Psi'(y) - \Psi(y) \\ &=& \frac{1}{2}\sigma^2y^2 + y\int_{-\infty}^0 \left(e^{yr}-\mathbb{I}_{\{|r| \leq 1\}}\right)r\Pi(dr) - \int_{-\infty}^0 \left(e^{yr}-1-yr\mathbb{I}_{\{|r| \leq 1\}}\right)\Pi(dr)  \\ &=& \frac{\sigma^2y^2}{2} + H(y),
\end{eqnarray*}
where the first equality follows from \eqref{eq:derivative-Legendre}, the second follows from \eqref{eq:psipp} and some straightforward algebra, and the third equality serves as a definition for the function $H$.  Observe that an integration by parts yields
\begin{eqnarray*}
H(y) &=&  \int_{-\infty}^0 \left(1-e^{yr}(1-yr)\right)\Pi(dr) \\
&=& - y^2\int_{-\infty}^0 e^{yr}r\Pi(-\infty,r)dr+ \left(1-e^{yr}(1-yr)\right)\Pi(-\infty,r)\Big|_{-\infty}^0\\
&=&-y^2\int_{-\infty}^0 e^{yr}r\Pi(-\infty,r)dr,
\end{eqnarray*}
where we used that $\lim\limits_{r\to-\infty}\Pi(-\infty,r)=0$ and $\lim\limits_{r\to 0}r^2\Pi(-\infty,r)=0$.
%Since,  $\Psi'(y) \stackrel{\infty}{\sim} \sigma^2 y$, see \eqref{eq:asympt_psip}, and, by integration again, $\Psi(y) \stackrel{\infty}{\sim} \frac{\sigma^2}{2} y^2$, we get  from the expressions of $H$, $\Psi$ and $\Psi'$, that  $\lim_{y\to \infty}\frac{2H(y)}{\sigma^2 y^2}=0$.
Finally, since
\begin{equation*}
\Psi_*''(\Psi'(y)) = \frac{1}{\Psi''(y)} \quad \text{and} \quad \Psi''(y) \stackrel{\infty}{\sim} \sigma^2,
\end{equation*}
we conclude that
\begin{equation*}
f_t(t\Psi'(y)) \stackrel{\infty}{\sim} \frac{1}{\sqrt{2\pi t}} \sqrt{\Psi_{*}''\left(\Psi'(y)\right)} e^{-t\Psi_*\left(\Psi'(y))\right)} \stackrel{\infty}{\sim}  \frac{1}{\sqrt{2\pi \sigma^2 t}} e^{-\frac{1}{2}t\sigma^2y^2+ty^2\int_{-\infty}^0 e^{yr}r\Pi(-\infty,r)dr}.
\end{equation*}
\bibliographystyle{abbrv}

%\bibliography{mybib}

\end{document}